\documentclass[a4paper,11pt]{amsart}

\usepackage[T1]{fontenc}
\usepackage[utf8]{inputenc}

\usepackage{lmodern}

\usepackage{dsfont}

\usepackage{enumerate}

\usepackage[left=2.6cm,right=2.6cm,top=2.7cm,bottom=2.7cm]{geometry}

\usepackage[foot]{amsaddr}
\title%
[Traces of magnetic Sobolev spaces]%
{Characterization of the traces on the boundary of functions in magnetic Sobolev spaces}
\author{Hoai-Minh Nguyen}
\address[H.-M. Nguyen]{\'Ecole Polytechnique F\'ed\'erale de Lausanne (EPFL)\\
SB, CAMA, Station 8\\
CH-1015 Lausanne\\ Switzerland}
\email{hoai-minh.nguyen@epfl.ch}

\author{Jean Van Schaftingen}

\address[J. Van Schaftingen]{Universit\'e catholique de Louvain (UCLouvain)\\ 
Institut de Recherche en Math\'ematique et Physique\\
Chemin du Cyclotron 2 bte L7.01.01\\
1348 Louvain-la-Neuve\\
Belgium}
\email{Jean.VanSchaftingen@uclouvain.be}
\thanks{J. Van Schaftingen was partially supported by the Projet de Recherche (Fonds de la Recherche Scientifique--FNRS) n. T.1110.14 ``Existence and asymptotic behavior of solutions to systems of semilinear elliptic partial differential equations''. J. Van Schaftingen acknowledges the hospitality of the EPFL where a substantial part of this work was performed.}



\usepackage{hyperref}

\usepackage[abbrev,backrefs]{amsrefs}

\usepackage{constants}
\usepackage[capitalize]{cleveref}

\usepackage{amssymb}
\usepackage{esint}

\theoremstyle{proposition}
\newtheorem{proposition}{Proposition}[section]
\newtheorem{theorem}[proposition]{Theorem}
\newtheorem{lemma}[proposition]{Lemma}

\theoremstyle{remark}
\newtheorem{remark}[proposition]{Remark}

\numberwithin{equation}{section}

\newcommand{\Rset}{\mathbb{R}}
\newcommand{\Cset}{\mathbb{C}}
\newcommand{\Nset}{\mathbb{N}}
\newcommand{\Sset}{\mathbb{S}}

\newcommand{\dif}{\,\mathrm{d}}
\newcommand{\st}{\,;\,}

\newcommand{\abs}[1]{\lvert #1 \rvert}
\newcommand{\bigabs}[1]{\bigl\lvert #1 \bigr\rvert}
\newcommand{\Bigabs}[1]{\Bigl\lvert #1 \Bigr\rvert}
\newcommand{\seminorm}[1]{\lvert #1 \rvert}
\newcommand{\norm}[1]{\lVert #1 \rVert}

\newcommand{\defeq}{\triangleq}
\newcommand{\compose}{\,\circ\,}

\newcommand{\intpot}[2][]{\mathcal{I}^{#1}_{#2}}

\DeclareMathOperator{\supp}{supp}

\DeclareMathOperator{\inj}{inj}
\DeclareMathOperator{\trace}{Tr}

\begin{document}

\begin{abstract} 
We characterize  the trace of magnetic Sobolev spaces defined in a half-space or in a smooth bounded domain in which the magnetic field $A$ is differentiable and its exterior derivative corresponding to the magnetic field $dA$ is bounded. 
In particular, we prove that, for $d \ge 1$ and $p>1$, the trace of the magnetic Sobolev space $W^{1, p}_A(\mathbb{R}^{d+1}_+)$ is exactly $W^{1-1/p, p}_{A^{\shortparallel}}(\mathbb{R}^d)$ where $A^{\shortparallel}(x) =( A_1, \dotsc, A_d)(x, 0)$ for $x \in \mathbb{R}^d$ with the convention $A = (A_1, \dotsc, A_{d+1})$ when $A \in C^1(\overline{\mathbb{R}^{d+1}_+},  \mathbb{R}^{d+1})$. 
We also characterize fractional magnetic Sobolev spaces as interpolation spaces and give extension theorems from a half-space to the entire space.
\end{abstract}

\maketitle

\section{Introduction}

The \emph{first-order magnetic Sobolev space} \(W^{1, p}_A (\Omega)\) on a given open set \(\Omega \subset \Rset^{d+1}\) with \(d \ge 1\)  is defined, for a given exponent \(p \in [1, +\infty)\), a vector field \(A \in C^1 (\Omega, \Rset^{d+1})\), as 
\citelist{\cite{Esteban_Lions_1989}\cite{Lieb_Loss_2001}\cite{Raymond_2017}\cite{Arioli_Szulkin_2003}\cite{Kato_1972}\cite{Avron_Herbst_Simon_1978}\cite{Cingolani_Secchi_2002}\cite{Sandier_Serfaty_2007}*{\S 1.1}}
\begin{equation}
\label{eq_finite_norm}
W^{1, p}_{A} (\Omega) \defeq  \left\{ U \in W^{1, 1}_{\mathrm{loc}} (\Omega, \Cset) \st  \Vert U\Vert_{W^{1, p}_A(\Omega)}^p\defeq \int_{\Omega} \abs{U}^p +  \abs{\nabla_A U}^p < + \infty \right\},
\end{equation}
where the weak covariant gradient \(\nabla_A U\) associated with $A$ of \(U \in W^{1, 1}_{\mathrm{loc}} (\Omega, \Cset)\) is defined as 
\footnote{Two opposite conventions are in use for the sign of the second term in \eqref{eq_def_covariant_derivative}, we follow \citelist{\cite{Esteban_Lions_1989}\cite{Lieb_Loss_2001}\cite{Raymond_2017}\cite{Arioli_Szulkin_2003}}, and we have thus opposite convention to \citelist{\cite{Kato_1972}\cite{Avron_Herbst_Simon_1978}\cite{Cingolani_Secchi_2002}\cite{Sandier_Serfaty_2007}*{\S 1.1}}. 
  The presence of two opposite conventions is related to charge of the particles that are studied; both conventions are equivalent up to complex conjugation. 
}
\begin{equation}
  \label{eq_def_covariant_derivative}
\nabla_A U = \nabla U + i A U \qquad \text{ in } \Omega. 
\end{equation}

Magnetic Sobolev spaces arise naturally for \(p = 2\) and \(d = 2\) (corresponding to \(\Omega \subseteq \Rset^3\)) in \emph{quantum mechanics} in the presence of a magnetic field described through its \emph{magnetic vector potential} \(A \in C^1 (\Omega, \Rset^3)\); the function \(U : \Omega \to \Cset\) is then a wave-function and the integral in \eqref{eq_finite_norm} is the quadratic form associated to the quantum mechanical Hamiltonian of a particle in a magnetic field (see, e.g.,  \citelist{\cite{Landau_Lifschitz_1977}*{Chapter XV}\cite{Gasiorowicz}*{Chapter 16}}).
In physical models, the only observable quantities are the \emph{magnetic field } \(B = \nabla \times A \simeq d A  \in C(\Omega, \bigwedge^2 \Rset^3) \) and the \emph{probability density} \(\abs{U}^2\).  Here and in what follows, \(dA\) denotes the exterior derivative of \(A\); for this, we consider \(A\) as an element in \(C^1(\Omega, \bigwedge^1 \Rset^{d+1})\).
The prevalent role of the magnetic field and of the probability density is reflected by the \emph{gauge invariance} invariance of the model:  the invariance of the Hamiltonian quadratic form defined by the left-hand side \eqref{eq_finite_norm} under a change of variables \(A \mapsto  A + \nabla \Phi\) and \(U \mapsto e^{-i\Phi} U\), for any phase shift \(\Phi \in C^1 (\Omega, \Rset)\), see, e.g., \cite[chapter 7]{Lieb_Loss_2001}. Geometrically, the invariant quantity \(i\,dA\) is the curvature  of the associated \(U (1)\)--connection (see for example \cite{Sontz_2015}*{Chapter 11}).

\medskip

Magnetic Sobolev spaces \(W^{1, p}_A(\Omega)\) generalize \emph{classical Sobolev spaces} \(W^{1, p} (\Omega)\), in  which \(A \equiv 0\),  defined by 
\begin{equation*}
 W^{1, p} (\Omega) \defeq  
 \biggl\{ U \in L^p(\Omega) \st  \Vert U\Vert_{W^{1, p}(\Omega)}^p\defeq  \int_{\Omega} \abs{U}^p +  \abs{\nabla U}^p < + \infty \biggr\}.
\end{equation*}
For \(0 < s < 1\) and \(1 \le p < +\infty\), the \emph{fractional Sobolev (Sobolev--Slobodecki{\u\i})  space} is defined as
\begin{equation*}
W^{s, p} (\partial \Omega) 
\defeq \left\{u \in L^p(\partial \Omega, \Cset) \st \Vert u \Vert_{W^{s, p}(\partial \Omega)}^p\defeq \Vert u \Vert_{L^p(\partial \Omega)}^p +  \seminorm{u}_{W^{s, p}(\partial \Omega)}^p < + \infty \right\}, 
\end{equation*}
where the \emph{Gagliardo seminorm} \(\seminorm{u}_{W^{s, p} (\partial \Omega)}\) of the function \(u : \partial \Omega \to \Cset\) is given by 
\begin{equation*}
\seminorm{u}_{W^{s, p} (\partial \Omega)}^p
 \defeq
  \iint\limits_{\partial \Omega \times \partial \Omega}
    \frac
      {\bigabs{u (y) - u (x)}^p}
      {\abs{y - x}^{d  + sp}}
      \dif x
      \dif y. 
\end{equation*}
When the set \(\Omega\) is bounded and its boundary is of class \(C^1\), or \(\Omega = \Rset^{d+1}_+ \defeq \big\{(x, t) \in \Rset^{d} \times \Rset; t > 0  \big\}\), and when \(p>1\), the trace theory is well known since Gagliardo's pioneer work  \cite{Gagliardo_1957} (see also \citelist{\cite{Uspenskii_1961}\cite{diBenedetto_2016}*{\S 10.17--10.18 and Proposition 17.1}\cite{Mironescu_Russ_2015}}).
The trace operator \(\operatorname{Tr}\) defined 
by 
\begin{equation*}
  \begin{aligned}
    \operatorname{Tr} : C^1(\Bar{\Omega}) &\to  C^1(\partial \Omega) \\[6pt]
    U & \mapsto  U \vert_{\partial \Omega},
  \end{aligned}
\end{equation*}
satisfies  for some positive constant \(C_{p, \Omega}\), for every \(U \in C^1 (\Omega)\) the estimate,
\begin{equation*}
\norm{\operatorname{Tr} U }_{W^{1-1/p, p}(\partial \Omega)} \le C_{p, \Omega} \Vert U \Vert_{W^{1, p}(\Omega)}, 
\end{equation*}
and therefore the linear operator \(\operatorname{Tr}\) extends to a bounded linear map from the Sobolev space \(W^{1, p}(\Omega)\) into fractional Sobolev space \(W^{1-1/p, p}(\partial \Omega)\).
Conversely there exists a bounded linear operator \(\operatorname{Ext} : W^{1 - 1/p, p} (\partial \Omega) \to W^{1, p} (\Omega)\) such that for any \(u \in W^{1-1/p, p}(\partial \Omega)\), 
\begin{equation*}
\operatorname{Tr} (\operatorname{Ext} u) = u \text{ on  } \partial \Omega \quad \text{ and } \quad 
\norm{\operatorname{Ext} u }_{W^{1, p}(\Omega)} \le C'_{p, \Omega} \Vert u \Vert_{W^{1-1/p, p}(\partial \Omega)}, 
\end{equation*}
for some positive constant \(C'_{p, \Omega}\) independent of \(u\). 
In particular, the map \(\operatorname{Tr}\) is surjective. Consequently,  the image under the trace operator of the space \(W^{1, p}(\Omega)\) is exactly the space \(W^{1-1/p, p}(\partial \Omega)\).  The space of traces can also be described as the real interpolations space \((W^{1, p} (\Rset^d), L^p (\Rset^d))_{1 - 1/p, p}\) in the framework of interpolation of Banach spaces \cite{Lions_Peetre_1964}*{Théorème VI.2.1}.

\medskip The trace theory for \(W^{1, p}_A(\Omega)\) can easily be derived from the one of  \(W^{1, p}_{A}(\Omega)\) when the magnetic potential \(A\) is bounded. In fact, by the triangle inequality
\begin{equation}
  \label{eq_igooT9zeesh0siequ}
\bigabs{\Vert\nabla_A u  \Vert_{L^p(\Omega)} - \Vert \nabla u \Vert_{L^p(\Omega)}}
\le \Vert A \Vert_{L^\infty(\Omega)} \norm{u}_{L^p (\Omega)}, 
\end{equation}
it follows that  \(W^{1, p}_A(\Omega) = W^{1,p}(\Omega)\) in this case. Hence the trace  of \(W^{1, p}_A(\Omega)\) is the space \(W^{1-1/p, p}(\partial \Omega)\) as well. 
The situation becomes more delicate when \(\Omega = \Rset^{d+1}_+\) and \(A\) is not assumed to be bounded but its total derivative \(D A\) or, even more physically, its exterior derivative \(dA\) is bounded. 
This type of assumption on \(A\) appears naturally in many problems in physics for which \(A\) is linear in simple settings. Moreover, even when \(A\) is bounded, the quantitative bounds resulting  \eqref{eq_igooT9zeesh0siequ} depend on \(\norm{A}_{L^\infty (\Omega)}\) which is not gauge invariant; it would be desirable to have estimates depending rather on \(dA\). To our knowledge,  a characterization of the trace of \(W^{1, p}_A(\Rset^{d+1}_+)\) is not known under  such assumption on  \(A\). The goal of this work is to give a complete answer to this question. 
Besides its own interest concerning boundary values in problems of calculus of variations and partial differential equations, this  is closely related to classes of fractional magnetic problems motivated by relativistic magnetic quantum physical models \cite{Ichinose_2013} that have been studied recently \citelist{\cite{dAvenia_Squassina_2018}\cite{Nguyen_Pinamonti_Squassina_2018} \cite{Nguyen_Squassina_2019} \cite{Squassina_Volzone_2016}\cite{Pinamonti_Squassina_Vecchi_2017}\cite{Wang_Xiang_2016}\cite{Liang_Repovs_Zhang_2018}\cite{Binlin_Squassina_Xia_2018}\cite{Ambrosio_2018}\cite{Ambrosio_dAvenia_2018}\cite{Fiscella_Pinamonti_Vecchi_2017}}.

\medskip 
Given \(0< s< 1\), \(1 \le p < +\infty\), and \(A^\shortparallel \in C (\Rset^{d}, \Rset^d)\), we define, for any measurable function \(u: \Rset^d \to \Cset\)
\begin{equation}
  \label{eq_def_mag_Gagliardo}
  \seminorm{u}_{W^{s, p}_{A^\shortparallel} (\Rset^d, \Cset)}^p \defeq 
  \iint\limits_{\Rset^d \times \Rset^d} \frac
      {\bigabs{e^{  i\,\intpot{A^\shortparallel} (x, y) } u (y) - u (x)}^p}
      {\abs{y - x}^{d + sp}}
      \dif x
      \dif y,
\end{equation}
where the potential \(\intpot{A^\shortparallel} : \Rset^d \times \Rset^d \to \Rset\) is defined for each \(x, y \in \Rset^d\) by 
\begin{equation*}
\intpot{A^\shortparallel} (x, y)\defeq \int_0^1 A^\shortparallel \big((1 - t) x + t y \big) \cdot (y - x) \dif t .
\end{equation*}
Here and in what follows \(\cdot\) denotes the complex scalar product. For \(A = (A_1, \dotsc, A_{d+1}) \in C(\overline{\Rset^{d+1}_+}, \Rset^{d+1})\), 
we  will consider its parallel component on the boundary \(A^{\shortparallel} : \Rset^d \to \Rset^d\) defined for each \(x \in \Rset^d\) by
\begin{equation}\label{def-A-para}
A^{\shortparallel} (x)= (A_1, \dotsc, A_d)(x, 0). 
\end{equation}

\medskip 
Our  first main result is 

\begin{theorem}
\label{theorem_intro_bounded} \label{thm1}
Let \(d  \ge 1\) and  \(1< p <+\infty\).  
There exists a positive constant \( C_{d, p}\) depending only on \(d\) and \(p\) such that 
if \(A \in C^1 (\overline{\Rset^{d + 1}_+},  \Rset^{d + 1})\) and 
\(\norm{dA}_{L^\infty (\Rset^{d + 1}_+)} \le \beta\), then 
 \begin{enumerate}[(i)]
   \item
   \label{thm1-1}
   for each \(U \in C^1_c(\overline{\Rset^{d+1}_+}, \Cset)\),
   \begin{equation*}
     \seminorm{U (\cdot, 0)}_{W^{1- 1/p, p}_{A^{\shortparallel}}(\Rset^d)} + \beta^{\frac{1}{2} - \frac{1}{2p}} \Vert 
   U (\cdot, 0) \Vert_{L^p(\Rset^d)} 
 \le  C_{d, p} \Bigl( \Vert\nabla_A U  \Vert_{L^p(\Rset^{d+1}_+)} 
 +  \beta^\frac{1}{2} \Vert U \Vert_{L^p(\Rset^{d+1}_+)} \Bigr),
\end{equation*}
\item for each \(u \in C^1_{c}(\Rset^d, \Cset)\),  there exists \(U \in C^1_c(\overline{\Rset^{d+1}_+})\) depending linearly on \(u\) such that 
\(U(x, 0) = u(x)\) in \(\Rset^{d}\) and 
\begin{equation*}
  \label{thm1-2}
 \norm{\nabla_A U}_{L^p(\Rset^{d+1}_+)} 
 +  \beta^\frac{1}{2} \Vert U \Vert_{L^p(\Rset^{d+1}_+)} \le C_{d, p} \Bigl( \seminorm{u}_{W^{1- 1/p, p}_{A^{\shortparallel}}(\Rset^d)} +   \beta^{\frac{1}{2} - \frac{1}{2 p}} \Vert u \Vert_{L^p(\Rset^d)} 
 \Bigr).   
\end{equation*}
\end{enumerate}
\end{theorem}

The conclusions of \cref{theorem_intro_bounded} are \emph{gauge-invariant}: all the functional norms are gauge-invariant and the constants only depend through \(\beta\) which is an upper bound of  the norm \(\norm{dA}_{L^\infty (\Rset^{d + 1}_+)}\) of the magnetic field on the half-space \(\Rset^{d + 1}_+\).

As a consequence of \cref{theorem_intro_bounded}, by a standard density argument (see Section~\ref{sect-characterization}), we obtain the following 
characterization of the trace of the space \(W^{1, p}_A(\Rset^{d+1}_+)\): 

\begin{theorem}
  \label{thm-2-intro} 
  Let \(d  \ge 1\) and  \(1 < p < +\infty\). Assume that 
  \(A \in C^1(\overline{\Rset^{d + 1}_+}, \bigwedge^1 \Rset^{d + 1})\) and that \(d A \in L^\infty (\Rset^{d + 1}_+, \bigwedge^2 \Rset^{d + 1})\).
  The trace mapping 
  \begin{equation*}
    \begin{aligned}
      \operatorname{Tr}: W^{1, p}_A (\Rset^{d+1}_+, \Cset) &\to  W^{1-1/p, p}_{A^{\shortparallel}}(\Rset^d, \Cset)\\[6pt]
      U(x, x_{d+1}) & \mapsto  U (x, 0)
    \end{aligned}
  \end{equation*}
  is linear and continuous. There exists a linear continuous mapping 
  \begin{equation*}
  \operatorname{Ext}: W^{1-1/p, p}_{A^{\shortparallel}}(\Rset^d, \Cset) \to W^{1, p}_A(\Rset^{d+1}_+, \Cset)
  \end{equation*}
  such that  \(\operatorname{Tr} \compose \operatorname{Ext}_{\Rset^{d + 1}_+}\) is the identity on \(W^{1-1/p, p}_{A^{\shortparallel}}(\Rset^d)\). Moreover, the corresponding estimates of \cref{thm1} with \(u = \operatorname{Tr} U\) and \(U = \operatorname{Ext}u\) are valid.
\end{theorem}

In the case where the magnetic field \(dA\) is constant, we obtain the following improvements:

\begin{theorem} \label{thm-3-intro} 
  Let \(d  \ge 1\) and  \(1< p <+\infty\).  Assume that  \(A \in C^1 (\overline{\Rset^{d + 1}_+},  \Rset^{d + 1})\) and that  \(dA\) is constant. We have, with \(u = \operatorname{Tr} U\) and \(U = \operatorname{Ext}u\), 
      \begin{equation}\label{thm3-p1}
      \seminorm{u}_{W^{1- 1/p, p}_{A^{\shortparallel}}(\Rset^d)} 
      + \norm{d A}^{\frac{1}{2} - \frac{1}{2p}} 
      \norm{u }_{L^p(\Rset^d)} 
      \le  C_{d, p} \Vert\nabla_A U  \Vert_{L^p(\Rset^{d+1}_+)} 
    \end{equation}
    and 
    \begin{equation}
      \label{thm3-p2}
      \norm{\nabla_A U}_{L^p(\Rset^{d+1}_+)} 
      +  \norm{dA}^\frac{1}{2} \Vert U \Vert_{L^p(\Rset^{d+1}_+)} \le C_{d, p} \seminorm{u}_{W^{1- 1/p, p}_{A^{\shortparallel}}(\Rset^d)} .   
    \end{equation}
  \end{theorem}


We later show that the space $W^{s, p}_{A^{\shortparallel}}(\Rset^d)$ with $0 < s < 1$ and $p\ge 1$ is the trace space of the space  $W^{1, p}_{A, 1- (1-s)p }(\Rset^d)$ whose definition is given in \eqref{def-WA-gamma}; moreover, the corresponding estimates hold (see Theorems~\ref{theorem_trace_halfplane_weight} and \ref{thm-3}).  

We establish similar estimates  for a smooth bounded domain  $\Omega \subset \Rset^{d + 1}$ and a magnetic potential \(A \in C^1 (\Bar{\Omega}, \Rset^{d+1})\). 
It is worth noting that the trace theory  in this setting is known as in the case $A \equiv 0$. Nevertheless, our estimates  (\cref{proposition_Bdd_dA_trace_domain} and \cref{proposition_Bdd_dA_extension_domain}) are gauge invariant, and sharpen estimates in the semi-classical limit
(\cref{proposition_semiclassical}). 

As a consequence of the trace theorems, we derive  a characterization of the space \(W^{s, p}_{A^\shortparallel} (\Rset^d, \Cset)\) as an interpolation space
(\cref{theorem_interpolation}). 
We also observe that the characterization of traces is also independent on the side of the hyperplane from which the trace is taken or to which the extension is made (this fact is not too trivial, see Remark~\ref{rem-extension}). Consequently,  the trace theorem provides an extension theorem from a half-space to the whole space (\cref{theorem_extension_hetero}). 

In an appendix, we show that under the assumption that some derivative of \(A\) is bounded, our magnetic fractional spaces have equivalent norms to other families of fractional spaces defined in the literature (\cref{proposition_Equivalent_norms}).

We now describe briefly the idea of the proof of the trace theory.  The proof of the trace estimates and of the construction of the extension is based on a standard strategy  that goes back to Gagliardo's seminal work \cite{Gagliardo_1957}. Concerning Theorem~\ref{theorem_intro_bounded} and its variants (Propositions~\ref{proposition_Bdd_dA_trace} and \ref{proposition_Bdd_dA_extension}), the key point of  our analysis lies on the observation that $A^{\shortparallel}$ defined in  $\Rset^d$ by \eqref{def-A-para} encodes the information the trace space of $W^{1, p}_A(\Rset^{d+1}_+)$ and an appropriate extension formula given in \eqref{eq_def_extension_Lip}. The proof of the trace estimates also involves Stokes theorems (Lemma~\ref{lem-Magnetic}) and a simple useful observation given in Lemma~\ref{lem-E-2}. Concerning Theorem~\ref{thm-3-intro} and its variants (Theorem~\ref{thm-3}), the new part is the trace estimates (see, e.g., \eqref{thm3-p1}). To this end, the Stoke formula and an averaging argument are used while taking into account the fact $dA$ is constant.  The proof for a domain $\Omega$ uses the results in the half space via local charts.


\section{Trace estimate for bounded magnetic field}

In this section, we prove the following trace estimate on the boundary of the half-space with a bounded magnetic field, which covers \eqref{thm1-1} in \cref{thm1}.

\begin{proposition}
\label{proposition_Bdd_dA_trace}
Let \(d  \ge 1\), \(0 < s<1\),  and \(1 \le p < +\infty\). 
There exists a positive constant \(C_{d, s, p}\) depending only on \(d\), \(s\) and \(p\) such that 
if  \(A \in C^1 (\overline{\Rset^{d + 1}_+},  \Rset^{d + 1})\), \(\norm{d A}_{L^\infty (\Rset^{d + 1}_+)} \le \beta\) and if \(U \in C^1_c(\overline{\Rset^{d+1}_+}, \Cset)\), then
\begin{equation}
  \label{pro-P1-state1}
|U(\cdot, 0)|_{W^{s, p}_{A^{\shortparallel}}(\Rset^d)}^p \le
C_{d, s, p}
          \iint\limits_{\Rset^{d + 1}_+}
            \frac{\bigabs{\nabla_A U (z, t)}^p + \beta^\frac{p}{2} \abs{U (z, t)}^p}{t^{1 - (1 - s)p}} \dif z \dif t 
\end{equation}
and 
\begin{equation}\label{pro-P1-state2}
\Vert U(\cdot, 0) \Vert_{L^p(\Rset^d)}^p \le 
C_{d, s, p}
  \Biggl(\iint\limits_{\Rset^{d + 1}_+}
  \frac{\bigabs{\nabla_A U (z, t)}^p }{t^{1 - (1 - s)p}} \dif t \dif z\Biggr)^{1 - s}
  \Biggl(\iint\limits_{\Rset^{d + 1}_+}
  \frac{\bigabs{U (z, t)}^p }{t^{1 - (1 - s) p}} \dif t \dif z\Biggr)^{s}.  
\end{equation}
\end{proposition}

As a consequence of \eqref{pro-P1-state2}, we have
\begin{equation}
  \beta^\frac{sp}{2}  \Vert U(\cdot, 0) \Vert_{L^p(\Rset^d)}^p \le 
  C'
  \iint\limits_{\Rset^{d + 1}_+}
  \frac{\bigabs{\nabla_A U (z, t)}^p + \beta^\frac{p}{2} \abs{U (z, t)}^p}{t^{1 - (1 - s)p}} \dif z \dif t .
\end{equation}


We first present several lemmas used in the proof of \cref{proposition_Bdd_dA_trace},
before going to the proof of \cref{proposition_Bdd_dA_trace} at the end of the section.

We define for each \(X, Y \in \overline{\Rset^{d+1}_+}\),  the homotopy operator
\begin{equation}
 \label{eq_def_intpot}
  \intpot{A} (X, Y)
 \defeq 
  \int_0^1 
    A \bigl((1 - t) X + t Y\bigr)
      \cdot (Y - X) 
      \dif t.
\end{equation}
We observe by integration by parts that 
\begin{equation}
  \label{eq_deriv_pot}
    D \intpot{A} (\cdot, Y) 
  = 
   - A.
 \end{equation}
Therefore, by the fundamental theorem of calculus, we have  
\begin{equation}\label{mvl-A}
    e^{i\,\intpot{A} (X, Y)} U (Y) - U (X)
  =
    \int_0^1 
      e^{i\,\intpot{A} (X, (1 - t) X + t Y )} 
      \nabla_A U \big((1 -t) X + t Y \big) \cdot (Y - X)
      \dif t.
\end{equation}

The following result will be used repeatedly in the present work. 

\begin{lemma} \label{lem-Magnetic}
  If \(d \ge 1\), \(A \in C^1 (\overline{\Rset^{d + 1}_+}, \Rset^{d + 1})\),
  then for every \(X, Y, Z \in \overline{\Rset^{d+1}_+}\), we have 
\begin{equation*}
  \intpot{A} (X, Y)
 + \intpot{A} (Y, Z)
 + \intpot{A} (Z, X)\\
 =
   \int_0^1 \int_0^{1 - s}
    d A \bigl((1 - t - s) X + t Y + s Z\bigr) [Y - X, Z - X]
    \dif t
    \dif s .
\end{equation*}
\end{lemma}

\begin{proof} We have if \(s, t \in [0, 1]\) and \(s + t \le 1\),
\begin{multline*}
d A \bigl((1 - t - s) X + t Y + s Z\bigr)[Y - X, Z - X]\\
= \frac{\mathrm{d}}{\mathrm{d} t} A \bigl((1 - t - s) X + t Y + s Z\bigr)[Z - X]
- \frac{\mathrm{d}}{\mathrm{d} s} A \bigl((1 - t - s) X + t Y + s Z\bigr)[Y - X]. 
\end{multline*}
Integrating with respect to $s, t \in [0, 1]$ with $s + t \le 1$ yields 
\begin{equation}
  \label{eq_ooxeiPeeNeiyo0Yio5r}
\begin{split}
\int_0^1 \int_0^{1 - s}&
d A \bigl((1 - t - s) X + t Y + s Z\bigr)[Y - X, Z - X]
\dif t
\dif s\\
= & \int_0^1 \int_0^{1 - s} \frac{\mathrm{d}}{\mathrm{d} t} A \bigl((1 - t - s) X + t Y + s Z\bigr)[Z - X] \dif t \dif s\\
&- \int_0^1 \int_0^{1 - t} \frac{\mathrm{d}}{\mathrm{d} s} A \bigl((1 - t - s) X + t Y + s Z\bigr)[Y - X] \dif s \dif t.
\end{split}
\end{equation}
By the fundamental theorem of calculus, for every \(s \in [0, 1]\) we have 
\begin{multline}
  \label{eq_Ahgu7shiengoo4gou6z}
\int_0^{1 - s} \frac{\mathrm{d}}{\mathrm{d} t} A \bigl((1 - t - s) X + t Y + s Z\bigr)[Z - X] \dif t\\
= A \bigl((1 - s) Y + s Z\bigr)[Z - X] - A \bigl((1 - s) X + s Z\bigr)[Z - X]
\end{multline}
and for every \(t \in [0, 1]\),  
\begin{multline}
  \label{eq_Shooth8Ah3ThooDoogh}
\int_0^{1 - t} \frac{\mathrm{d}}{\mathrm{d} s} A \bigl((1 - t - s) X + t Y + s Z\bigr)[Y - X] \dif s \\
= A \bigl(t Y + (1 -t) Z\bigr)[Y - X] - A \bigl((1 - t) X + t Y\bigr)[Y - X].
\end{multline}
By inserting \eqref{eq_Ahgu7shiengoo4gou6z}, \eqref{eq_Shooth8Ah3ThooDoogh} and \eqref{eq_ooxeiPeeNeiyo0Yio5r} 
and by applying the change of variable \(s = t\), \(t = 1 -s\), we obtain
\begin{equation*}
  \begin{split}
    \int_0^1 \int_0^{1 - s} &
    d A \bigl((1 - t - s) X + t Y + s Z\bigr) [Y - X, Z - X]
    \dif t
    \dif s\\
    &= \int_0^1 
        A \bigl((1 - s) Y + s Z\bigr)[Z - X] \dif s - \int_0^1  A \bigl(t X + (1-t) Z\bigr)[Z - X] \dif t \\
        &\qquad - \int_0^1  A \bigl((1 - s) Y + s Z\bigr)[Y - X] \dif s
        + \int_0^1 A \bigl((1 - t) X + t Y\bigr)[Y - X]\dif t\\
    &= \intpot{A}(Z, X) + \intpot{A} (Y, Z) + \intpot{A}(X, Y),
    \end{split}
\end{equation*}
in view of the definition in \eqref{eq_def_intpot}.
\end{proof}

Using \cref{lem-Magnetic}, we can establish the following simple  result which is the key ingredient of the proof of Proposition~\ref{proposition_Bdd_dA_trace}. 

\begin{lemma} \label{lem-E-2} 
  If \(d \ge 1\), \( U \in C^1(\overline{\Rset^{d+1}_+}, \Cset)\) and \(A \in C^1 (\overline{\Rset^{d + 1}_+},  \Rset^{d + 1})\), then for every \(X, Y, Z \in \overline{\Rset^{d+1}_+}\), we have 
 \begin{multline*}
   \bigabs{e^{i\,\intpot{A} (X, Y)} U (Y) 
     - U (X)}
 \le 
    \bigabs{
        e^{i\,\intpot{A} (Z, Y)} U (Y) 
      -  
      U (Z)}
    + \bigabs{
        e^{i\,\intpot{A} (Z, X)} U (X)
         - U (Z)
      }\\
   +
   \abs{U (Z)} \min\left\{1, \tfrac{1}{2} \norm{dA}_{L^\infty} \abs{ (X - Z) \wedge (Y - Z)} \right\}.
\end{multline*}
\end{lemma}
\begin{proof}
Since
\[
\bigabs{
  e^{i\, \intpot{A} (X, Y)} U (Y) 
  - 
  e^{i\, (\intpot{A} (X, Y) + \intpot{A} (Y, Z) )} U (Z)}
= 
\bigabs{
  e^{i\,\intpot{A} (Z, Y)} U (Y) 
  -  
  U (Z)},  
\]
and 
\[
\bigabs{
  U (X) 
  - 
  e^{i\,\intpot{A} (X, Z)}
  U (Z)
}
=
\bigabs{
  e^{i\,\intpot{A} (Z, X)} U (X) 
  -  
  U (Z)}, 
\]
by the triangle inequality, we obtain
\begin{multline*}
  \bigabs{e^{i\,\intpot{A} (X, Y)} U (Y) 
    - U (X)}
  \le 
  \bigabs{
    e^{i\,\intpot{A} (Z, Y)} U (Y) 
    -  
    U (Z)}
  + \bigabs{
    e^{i\,\intpot{A} (Z, X)} U (X)
    - U (Z)
  }\\
  +
  \bigabs{e^{i\, (\intpot{A} (X, Y) + \intpot{A} (Y, Z) )} U (Z) -  e^{i\,\intpot{A} (X, Z)}
    U (Z)}.
\end{multline*}
We observe that 
\begin{multline*}
  \bigabs{e^{i\, (\intpot{A} (X, Y) + \intpot{A} (Y, Z) )} U (Z) -  e^{i\,\intpot{A} (X, Z) U (Z)}
    U (Z)}
  = 
  \bigabs{e^{i\, (\intpot{A} (X, Y) + \intpot{A} (Y, Z) + \intpot{A} (Z, X))} - 1}\abs{U (Z)}. 
\end{multline*}
and conclude with \cref{lem-Magnetic}.  
\end{proof}

\begin{proof}
[Proof of \cref{proposition_Bdd_dA_trace}]
\resetconstant
Applying \cref{lem-E-2}, we have, for each \(x, y \in \Rset^d\) and with the notations  \(x \simeq (x, 0) \in \Rset^{d + 1}_+\), \(y \simeq (y, 0) \in \Rset^{d + 1}_+\),  \(Z  = (\frac{x+ y}{2}, \abs{y - x})\) and \(u = U (\cdot, 0)\),
\begin{multline}\label{eq_Vooph8jieH}
 \abs{
    e^{
          i\,\intpot{A} (x, y)
      }     
      u (y) 
  - 
    u (x)
    }\\[6pt]
  \le 
  \bigabs{
      e^{
          i\,\intpot{A} (Z, x)} U(x)- 
 U (Z)}    + \bigabs{
      e^{
          i\,\intpot{A} (Z, y) } U (y) - U (Z)} +  \abs{U (Z)} \norm{dA}_{L^\infty}^{1/2} \abs{y - x}.
\end{multline}
Using \eqref{mvl-A}, we derive from \eqref{eq_Vooph8jieH} that 
\begin{multline}\label{eq_Vooph8jieH-1}
\abs{e^{  i\,\intpot{A} (x, y)} u (y) - u (x)} \le 
        \abs{y - x}
        \int_0^1 
          \bigabs{\nabla_A U \big( (1-t) x + t Z \big)}
          \dif t 
      \\
 + \abs{y - x}
        \int_0^1 
          \bigabs{\nabla_A U \big( (1-t) y + t Z \big)}
          \dif t  
          + \C \beta^\frac{1}{2} \abs{y - x} 
    \bigabs{U(Z)}. 
    \end{multline}
    Here and in what follows in this proof, \(C_1, C_2, \dotsc\) denote positive constants depending only on \(d, p\), and \(s\).

Since \(0 < s < 1\), 
using the fact, for a measurable function \(f\) defined on \([0, 1]\),  by H\"older's inequality, that 
\begin{equation*}
\left(\int_0^1 |f(t)| \dif t \right)^{p} \le \C \int_0^1 t^{(1-s)(p-1)}|f(t)|^p \dif t, 
\end{equation*}
we derive from \eqref{eq_Vooph8jieH-1} that, since \(Z = (\frac{x + y}{2}, \abs{y - x})\),
\begin{multline}
\label{eq_looPhaWoo5}
\frac{\abs{e^{  i\,\intpot{A^\shortparallel} (x, y)} U (y) - U (x)}}{\abs{y - x}^{d + sp}} \le 
\C\Biggl(
         \int_0^1 t^{(1-s)(p-1)}
         \frac{\bigabs{\nabla_A U \big( (1 - \tfrac{t}{2}) x + \tfrac{t}{2} y, t \abs{y - x} \big)}^p}{\abs{y - x}^{d - (1 - s) p}} 
          \dif t 
      \\[6pt]
      \shoveright{+ 
        \int_0^1 t^{(1-s)(p-1)} 
        \frac{\bigabs{\nabla_A U \big( (1 - \tfrac{t}{2}) y + \tfrac{t}{2} x, t \abs{y - x})\big)}^p}{{\abs{y - x}^{d - (1 - s) p}}}
        \dif t } \\
          +  \beta^\frac{p}{2}\frac{\bigabs{U(\frac{x + y}{2}, \abs{y - x})}^p}{\abs{y - x}^{d - (1 - s) p}} 
  \Biggr). 
\end{multline}

We now estimate the integral of the left-hand side of \eqref{eq_looPhaWoo5} with respect to \(x\) and \(y\) by estimating the integrals of the three terms on the right-hand side. 
For every \(t \in (0, 1)\), making the change of variable \(\eta = (1 - \frac{t}{2}) x + \tfrac{t}{2} y\) and \(\xi = t (x - y)\),  we obtain for every \(t \in (0, 1)\),
\begin{multline}\label{pro-P1-p0}
  t^{(1-s)(p-1)}  \iint\limits_{\Rset^d \times \Rset^d}      
      \frac
      {\bigabs{\nabla_A (U (1 - \tfrac{t}{2}) x + \tfrac{t}{2} y, t \abs{y - x} \big)}^p}
        {\abs{y - x}^{d - (1 - s) p}}
        \dif x \dif y\\
        = \frac{1}{t^{1 - s}} \iint\limits_{\Rset^d \times \Rset^d}      
      \frac
        {\bigabs{\nabla_A U (\eta, \abs{\xi})}^p}
        {\abs{\xi}^{d - (1 - s) p}}
        \dif \eta \dif \xi 
        =  \frac{|\Sset^{d-1} |}{t^{1 - s}} \iint\limits_{\Rset^d \times (0, + \infty)}      
      \frac
        {\bigabs{\nabla_A U (\eta, r)}^p}
        {r^{1- (1 - s) p}}
        \dif \eta \dif  r . 
\end{multline}
Since \(0 < s < 1\), it follows that  
\begin{multline}\label{pro-P1-p1}
 \iint\limits_{\Rset^d \times \Rset^d}     \int_0^1 t^{(1-s)(p-1)}
 \frac{\bigabs{\nabla_A \bigl(U (1 - \tfrac{t}{2}) x + \tfrac{t}{2} y, t \abs{y - x} \big))}^p}{\abs{y - x}^{d - (1 - s) p}} 
          \dif t  \dif x \dif y  \\[-1em]
          \le \Cl{cst_ahWuaHuuz5ohgaih}  \iint\limits_{\Rset^d \times (0, + \infty)}      
      \frac
        {\bigabs{\nabla_A U (x, r)}^p}
        {r^{1- (1 - s) p}}
        \dif x  \dif  r. 
\end{multline}
Similarly, we have
\begin{multline}\label{pro-P1-p2}
 \iint\limits_{\Rset^d \times \Rset^d}     \int_0^1 t^{(1-s)(p-1)}
 \frac{\bigabs{\nabla_A \bigl(U (1 - \tfrac{t}{2}) y + \tfrac{t}{2} x, t \abs{y - x} \big))}^p}{\abs{y - x}^{d - (1 - s) p}} 
          \dif t  \dif x \dif y \\[-1em]
          \le \Cr{cst_ahWuaHuuz5ohgaih}   \iint\limits_{\Rset^d \times (0, + \infty)}      
      \frac
        {\bigabs{\nabla_A U (x, r)}^p}
        {r^{1- (1 - s) p}}
      \dif x    \dif  r.
\end{multline}
Using  the change of variable \(\eta = \frac{x + y}{2}\) and 
\(\xi = y - x\),  and the polar coordinates, by the same way to obtain \eqref{pro-P1-p0},  we also reach  
\begin{equation}\label{pro-P1-p3}
    \iint\limits_{\Rset^d \times \Rset^d}
    \frac{\bigabs{U (\frac{x + y}{2}, \abs{y - x})}^p}{\abs{x - y}^{d - (1 - s) p}}
      \dif x
      \dif y 
      =   \abs{\Sset^{d - 1}}\iint\limits_{\Rset^d \times (0, +\infty)}
      \frac{ \bigabs{U (\eta, r)}^p}{r^{1 - (1 - s) p}}
      \dif \eta
      \dif r.
\end{equation}

Combining  \eqref{eq_looPhaWoo5}, \eqref{pro-P1-p1}, \eqref{pro-P1-p2}, and \eqref{pro-P1-p3} yields  
\begin{equation}
\label{eq_voh2quoo7D}
\seminorm{u}_{W^{s, p}_{A^{\shortparallel}}(\Rset^d)}^p 
  \le \C
          \iint\limits_{\Rset^{d + 1}_+}
            \frac{\bigabs{\nabla_A U (z, t)}^p + \norm{dA}_{L^\infty}^\frac{p}{2} \abs{U (z, t)}^p}{t^{1 - (1 - s)p}}  \dif z \dif t, 
\end{equation}
which is \eqref{pro-P1-state1}. 


We next prove \eqref{pro-P1-state2}.  Since, by the diamagnetic inequality \(\abs{\nabla \abs{U}} \le \abs{\nabla_A U}\) in \(\Rset^{d + 1}_+\), 
we have,  for  \(x \in \Rset^d\) and \(t \ge 0\),
\begin{equation*}
  \abs{U(x, 0)}
  \le \abs{U(x, t)} + \int_0^t \abs{\nabla_A U(x, s)} \dif s.  
\end{equation*}
It follows that, for \(x \in \Rset^d\) and \(\lambda  > 0\), 
\[
    \abs{U (x, 0)} 
  \le 
      \int_{0}^\lambda \abs{\nabla_A U (x, t)} \dif t
    +
      \frac
        {2}
        {\lambda}
      \int_{\lambda/2}^\lambda
        \abs{U (x, t)}
        \dif t.
\]
Using H\"older's inequality, we deduce that 
\begin{equation}
\label{eq_Ua6eeFeQui}
\begin{split}
    \Biggl(
    \int_{\Rset^d} 
      \abs{U (x, 0)}^p \dif x
      \Biggr)^\frac{1}{p}
  \le & 
  \int_{0}^\lambda
    \Biggl(
      \int_{\Rset^d} \abs{\nabla_A U (x, t)}^p \dif x
      \Biggr)^\frac{1}{p}\dif t
  + 
  \frac{2}{\lambda} 
  \int_{\frac{\lambda}{2}}^\lambda
  \Biggl(
  \int_{\Rset^d} \abs{U (x, t)}^p \dif x
  \Biggr)^\frac{1}{p}
  \dif t \\[6pt]
  \le & 
  \C \,\lambda^s
  \Biggl(\ 
  \iint\limits_{\Rset^{d+1}_+} \frac{\abs{\nabla_A U (x, t)}^p}{t^{1 - (1 - s)p}} \dif x \dif t
  \Biggr)^\frac{1}{p}
  + 
  \frac{\C }{\lambda^{1 - s}}
  \Biggl(\ 
  \iint\limits_{\Rset^{d+1}_+ } \frac{\abs{U (x, t)}^p}{t^{1 - (1 - s)p}} \dif x \dif t
  \Biggr)^\frac{1}{p} 
  . 
\end{split}
\end{equation}
Optimizing with respect to \(\lambda > 0\), we obtain 
\begin{equation}
  \label{eq_HaT6soayai}
  \Biggl(
  \int_{\Rset^d} 
  \abs{U (x, 0)}^p \dif x
  \Biggr)^\frac{1}{p}
  \le \C\Biggl(\ 
  \iint\limits_{\Rset^{d+1}_+} \frac{\abs{\nabla_A U (x, t)}^p}{t^{1 - (1 - s)p}} \dif x \dif t
  \Biggr)^\frac{1-s}{p}  \Biggl(\ 
  \iint\limits_{\Rset^{d+1}_+ } \frac{\abs{U (x, t)}^p}{t^{1 - (1 - s)p}} \dif x \dif t
  \Biggr)^\frac{s}{p}, 
\end{equation}
which is \eqref{pro-P1-state2}.  
\end{proof}

In what follows $B(x, R)$ denotes the open ball in $\Rset^d$ centered at $x$ and of radius $R$; when $x = 0$, one uses the notation $B_R$ instead.  Using the same arguments, we obtain  a localized version of \cref{proposition_Bdd_dA_trace} which will be used in Section~\ref{sect-domain}.

\begin{proposition}
  \label{proposition_Bdd_dA_trace_ball}
  Let \(d \ge 1\), \(0 < s < 1\) and \(1 \le p < +\infty\).
  There exists a positive  constant \(C_{d, s, p}\) depending only on $s, p$, and $d$ such that if \(R \in (0, + \infty)\), \(A \in C^1 (B (0, R) \times [0, R],  \Rset^{d + 1})\), if \(\norm{dA}_{L^\infty} \le \beta\) and if \(U \in C^\infty (B [0, R] \cap \Rset^{d + 1}_+)\) and \(u = U (\cdot, 0)\), then 
  \begin{equation*}
    \iint\limits_{B (0, R) \times B (0, R)}
    \hspace{-1em}
    \frac{\abs{
        e^{
          i\,\intpot{A^\shortparallel} (x, y) 
        } 
        U (y, 0) 
        - 
        U (x, 0)
      }^p}{\abs{y - x}^{d + s p}}
    \dif x \dif y
    \\
    \le
    C_{d, s, p}
        \hspace{-1.3em}
    \iint\limits_{B (0, R) \times [0, R]}
    \hspace{-1em}
    \frac{\bigabs{\nabla_A U (z, t)}^p + \beta^\frac{p}{2} \abs{U (z, t)}^p}{t^{1 - (1 - s)p}} \dif t \dif z
    .
  \end{equation*}
  \end{proposition}

\section{Extension to the half-space}

In this section, we prove the following extension result which implies \eqref{thm1-2} of  \cref{thm1}. 

\begin{proposition}
\label{proposition_Bdd_dA_extension}
Let \(d  \ge 1\),  \(0 < s < 1\) and \(1\le p < +\infty\). 
There exists a positive constant \( C_{d, s, p}\) depending only on \(d\), \(s\), and \(p\) such that for
every \(A \in C^1 (\overline{\Rset^{d + 1}_+},  \Rset^{d + 1})\) with \(\norm{d A}_{L^\infty (\Rset^{d + 1}_+)} \le \beta\)
and for any \(u \in C^1_{c}(\Rset^d, \Cset)\) with compact support, one can find \(U \in C^1_c (\overline{\Rset^{d+1}_+})\) depending linearly on \(u\) such that 
\(U(x, 0) = u(x)\) in \(\Rset^{d}\), 
\begin{equation*} 
   \iint\limits_{\Rset^{d + 1}_+}
   \frac{\abs{\nabla_A U (x, t)}^p}
      {t^{1 - (1 - s)p}} \dif x \dif t
  \le 
  C_{d, s, p}
  \Bigl( \seminorm{u}_{W^{s, p}_{A^{\shortparallel}}(\Rset^d)}^p 
  + \beta^\frac{sp}{2}\norm{u}_{L^p(\Rset^d)}^p  \Bigr)
\end{equation*}
and 
\begin{equation*}
  \iint\limits_{\Rset^{d + 1}_+}
  \frac{  \abs{U (z, t)}^p}
  {t^{1 - (1 - s)p}} \dif x \dif t
  \le 
  \frac{C_{d, s, p}}{\beta^\frac{(1 - s)p}{2}}
  \norm{u}_{L^p(\Rset^d)}^p. 
\end{equation*}
\end{proposition}

\begin{proof}
\resetconstant%
  Let \(\varphi \in C^\infty_c (\Rset^d)\)  and \(\theta \in C^\infty(\Rset)\) be such that 
  \begin{gather*}
\int_{\Rset^d} \varphi = 1, \quad  \varphi (x) = 0 \text{ for }  \abs{x} > 1,\\[1em]
\theta = 1 \text{ in } (-a/2, a/2), \quad  \theta = 0 \text{ in } \Rset \setminus (-a, a), \quad \text{ and } \quad  |\theta'| \le \Cl{cst_air1aiGheiCh0soo1}/a \text{ in \(\Rset\)},
\end{gather*}
with \(a = \beta^{-1/2}\) for some  positive constant \(\Cr{cst_air1aiGheiCh0soo1}\) independent of \(a\), and set, for \(t > 0\) 
\begin{equation*}
\varphi_t (\cdot ) \defeq t^{-d} \varphi(\cdot /t) \quad \text{ in }  \Rset^d. 
\end{equation*}

We define the function \(U : \Rset^{d + 1}_+ \to \Cset\) by setting, for every \((x, t) \in \Rset^d \times [0, +\infty)\),  
\begin{equation}
\label{eq_def_extension_Lip}
U (x, t) \defeq \theta(t)  \int_{\Rset^d}
\varphi_t \,(x - y) \, e^{i  \intpot{A} ((x, t), y)} \, u (y) \dif y,
\end{equation}
where we have identified the point \((y, 0) \in \Rset^{d} \times \{0\} \subset \Rset^{d + 1}\) with \(y \in \Rset^d\).

By \eqref{eq_deriv_pot},  for every \((x, t) \in \Rset^{d + 1}_+\), we have,  for \(1 \le j \le d\), 
\begin{equation*}
\partial_j U(x, t) + i A_j (x, t) U (x, t) = \theta (t) \int_{\Rset^d} 
\partial_j \varphi_t \,(x - y) \, e^{i  \intpot{A} ((x, t), y)}  \, u (y) \dif y
\end{equation*}
and 
\begin{multline*}
  \partial_{d + 1} U(x, t) + i A_{d + 1} (x, t) U (x, t) \\
  = 
  \int_{\Rset^d} \biggl(\theta'(t)\varphi_t\,(x - y) - \theta (t) \Bigl(\frac{d}{t}
    \varphi_t\,(x - y)
    -\frac{\theta(t)}{t^{d + 2}} \nabla  \varphi \, \Big(\frac{x - y}{t} \Bigr) \cdot (x - y)  \Bigr) \biggr) 
    \, e^{i  \intpot{A} ((x, t),y)} \, u (y) \dif y.
\end{multline*}
With the notations
\(\Phi^i_t (z) \defeq \partial_j \varphi_t (z)\) if \(1 \le j \le d\) and \(\Phi^{d + 1}_t (z) =  - \frac{d}{t} \varphi_t (z) -t^{-(d + 2)}\nabla \varphi (z/t) \cdot z\) for \(z \in \Rset^d\),  we have, since \(\int_{\Rset^d} \Phi^i_t =0\), 
\begin{multline}\label{computation-2}
\int_{\Rset^d}
\Phi^i_t \,(x - y) \, e^{i  \intpot{A} ((x, t), y)}  \, u (y) \dif y \\
= \int_{\Rset^d}
\Phi^i_t (x - y) \, \big(e^{i  \intpot{A} ((x, t), y)} - e^{i  \intpot{A} (x, y) + i  \intpot{A} ((x, t), x)}  \big)  \, u (y)   \dif y \\
+ \int_{\Rset^d}
\Phi^i_t (x - y) e^{ i\,\intpot{A} ((x, t), x)}  \Bigl( e^{i  \intpot{A^\shortparallel} (x, y) } u(y) - u(x) \Bigr) \dif y.
\end{multline}
It follows that,  for every \((x, t) \in \Rset^{d + 1}_+\),
\begin{equation}\label{computation-3-1}
  \abs{\nabla_A U (x, t)} \le \C \,\bigl(L_1(x, t) + L_2(x, t) +  L_3(x, t) \bigr), 
\end{equation}
where the functions \(L_1, L_2, L_3 : \Rset^{d + 1}_+ \to \Rset\) are defined for each \((x, t) \in \Rset^{d + 1}_+\) by
\begin{align*}
  \displaystyle L_1 (x, t)&\defeq    \frac{\mathds{1}_{(0, a)}(t)}{t^{d +1} } \int_{B(x, t)} \Bigabs{e^{i  \intpot{A^\shortparallel} (x, y)} u(y) - u(x) } \dif y, \\
  \displaystyle L_2(x, t)&\defeq   \frac{\mathds{1}_{(0, a)}(t)}{t^{d+1}} \int_{B(x, t)}  \Bigabs{e^{i  ( \intpot{A} ((x, t), y ) + \intpot{A^{\shortparallel}} (y , x) + \intpot{A} (x, (x, t)))} - 1 }  \, 
  \abs{u (y)} \dif y,  \\
\displaystyle L_3 (x, t)&\defeq   \frac{\mathds{1}_{(0, a)} (t)}{a\, t^d}  \int_{B (x, t)} \abs{u(y)} \dif y.  
\end{align*}

We first have,  by H\"older's inequality,
\begin{equation}\label{computation-4}
  \begin{split}
  \iint\limits_{\Rset^{d+1}_+}  \frac{\abs{L_1(x, t)}^p}{t^{1 - (1-s) p}}  \dif x \dif t  
& 
\le \C 
\iint\limits_{\Rset^{d+1}_+} \frac{1}{t^{1  + d + s p}} \biggl(\int_{B(x, t)} |e^{i  \intpot{A^\shortparallel} (x, y)} u(y) - u(x)|^p \dif y \biggr) \dif x \dif t \\
& \le \C \iint\limits_{\Rset^d \times \Rset^d} \frac{|e^{i  \intpot{A^\shortparallel} (x, y)} u(y) - u(x)|^p}{\abs{y - x}^{d + sp}} \dif x \dif y.  
\end{split}
\end{equation}

Next, by \cref{lem-Magnetic}, if \(x, y \in \Rset^d\) and \(t \in (0, +\infty)\), we have  
\begin{equation*}
  \bigabs{ \intpot{A} ((x, t), y) + \intpot{A} (y, x) + \intpot{A} (x, (x, t))} 
  \le 
  \C 
  \,
  \norm{d A}_{L^\infty (\Rset^{d + 1}_+)} t^2,  
\end{equation*}
and therefore, \(\abs{y - x} \le t\),  
\begin{equation*}
  \bigabs{e^{i (\intpot{A} ((x, t), y) + \intpot{A} (y, x) + \intpot{A} (x, (x, t)) )} - 1 } 
  \le 
  \C
  \,
  \beta^{1/2} 
  t. 
\end{equation*}
It follows from  Fubini's theorem that 
\begin{equation}\label{computation-5}
  \begin{split}
    \iint\limits_{\Rset^{d+1}_+} \frac{\abs{L_2(x, t)}^p}{ t^{1 - (1-s) p}} \dif x \dif t  
  &\le \C \iint\limits_{\Rset^{d} \times (0, a)}   \frac{\beta^{p/2}}
    {t^{d + 1 - (1 - s)p} }
    \int\limits_{B(x, t)} |u(y)|^p \dif y \dif x \dif t \\
    & = \C \int_0^a  \frac{\beta^{\frac{p}{2}}}
    {t^{1 - (1 - s)p} } \int_{\Rset^d} \abs{u (y)}^p \dif y \dif t
    = \C  \, \beta^{\frac{sp}{2}}  \int_{\Rset^d }  |u(y)|^p \dif y.
  \end{split}
\end{equation}
Finally, we have 
\begin{equation} 
  \label{computation-6} 
  \iint\limits_{\Rset^{d+1}_+}  \frac{\abs{L_3(x,  t)}^p}{a^p\, t^{1 - (1-s) p}} \dif x \dif t 
  \le \C \int\limits_{\Rset^d \times (0, a)} \int\limits_{B (x, t)} \abs{u (y)}^p \dif y \dif x \dif t \le \C \, \beta^\frac{sp}{2}   \int_{\Rset^d} \abs{u}^p.
\end{equation}
Combining \eqref{computation-3-1}, \eqref{computation-4},  \eqref{computation-5}, and \eqref{computation-6} yields 
\begin{equation}\label{computation-7}
\iint\limits_{\Rset^{d+1}_+} \frac{ \abs{\nabla_A U (x, t)}^p}{t^{1 - (1-s) p}} \dif x \dif t 
  \le 
    \C 
    \,
    \seminorm{ u }_{W^{s, p}_{A^{\shortparallel}}(\Rset^d)}^p 
    + 
    \C 
   \,
    \beta^\frac{sp}{2} 
    \norm{u}_{L^p(\Rset^d)}^p. 
\end{equation}

Similar to \eqref{computation-6}, we also have 
\begin{equation}\label{computation-8}
  \iint\limits_{\Rset^{d+1}_+} \frac{ \abs{U (x, t)}^p}{t^{1 - (1-s) p}} \dif x \dif t \le \C 
  \frac{\norm{u}_{L^p(\Rset^d)}^p}{\beta^{\frac{(1 - s)p}{2}}}.
\end{equation}
The conclusion now follows from \eqref{computation-7} and \eqref{computation-8}. 
\end{proof}

\begin{remark} 
 If in the proof of \cref{proposition_Bdd_dA_extension}, one takes \(a = 1\), then the following estimate holds: 
\begin{equation*}
   \iint\limits_{\Rset^{d + 1}_+}
      \frac{\abs{\nabla_A U (x, t)}^p + \abs{U (x, t)}^p}
      {t^{1 - (1 - s)p}} \dif x \dif t\\
  \le 
  C
  \Bigl( \seminorm{u}_{W^{s, p}_{A^{\shortparallel}}(\Rset^d)}^p 
  +  \bigl( 1+ \norm{dA}_{L^\infty}^\frac{p}{2}  \bigr) \norm{u}_{L^p(\Rset^d)}^p  \Bigr).
\end{equation*}
\end{remark}

We also have a local version of \cref{proposition_Bdd_dA_extension}.

\begin{proposition}
 \label{proposition_Bdd_dA_extension_local}
 Let \(d \ge 1\), \(0 < s < 1\) and \(1 \le p < +\infty\).
 There exists a constant \(C_{d, s, p}\) such that for every \(u \in C^\infty (B (0, 2R), \Cset)\) and every \(A 
 \in C^1 (\Bar{B}_{2 R} \times [0, R], \bigwedge^1 \Rset^{d + 1})\) such that \(\norm{dA}_{L^\infty (B(0, 2R) \times [0, R])} + \frac{1}{R^2} \le \beta\), there exists
 \(U \in C^\infty (B (0, R) \times (0, R), \Cset)\) such that \(U (\cdot, 0) = u (\cdot)\) on \(B (0, R)\)
and  
 \begin{multline*}
   \iint\limits_{B (0, R) \times (0, R)}
   \frac{\abs{\nabla_A U (x, t)}^p}
   {t^{1 - (1 - s)p}} \dif x \dif t\\[-1em]
   \le 
   C_{d, s, p}
   \Biggl(\; 
   \iint\limits_{B(0, 2R) \times B (0, 2R)}
   \!\!\!
   \frac{\bigabs{e^{i\,\intpot{A^\shortparallel} (x, y)} u (y) - u (x)}^p}
   {\abs{x - y}^{d + sp}}
   \dif x 
   \dif y
   + \beta ^\frac{sp}{2}
   \int_{B(0, 2R)} \abs{u}^p
   \Biggr)
 \end{multline*}
 and
 \begin{equation*}
   \int_{B (0, R) \times (0, R)} 
   \frac{\abs{U (z, t)}^p}{t^{1 - (1 - s)p}}
   \dif t \dif z
   \le 
   \frac{C_{d, s, p}}{\beta^\frac{(1 - s)p}{2}}
   \int_{B (0, 2R)}
   \abs{u}^p.
 \end{equation*}
\end{proposition}

\section{Characterizations of trace spaces}\label{sect-characterization}

For \(\gamma \in \Rset\), we define the weighted space
\begin{equation}\label{def-WA-gamma}
  W^{1, p}_{A, \gamma} (\Rset^{d + 1}_+) \defeq  \left\{ u \in W^{1, 1}_{\mathrm{loc}} (\Rset^{d +1}_+, \Cset)
\st  \Vert U\Vert_{W^{1, p}_{A, \gamma} (\Omega)} < +\infty \right\}, 
\end{equation}
where
\[ \Vert U\Vert_{W^{1, p}_{A, \gamma} (\Omega)}^p
\defeq 
\iint\limits_{\Rset^{d + 1}_+} \Bigl(\abs{U (x, t)}^p +  \abs{\nabla_A U (x, t)}^p \Bigr) t^\gamma \dif x \dif t
\]
It is standard to check that the space  \(W^{1, p}_{A, \gamma}(\Rset^{d+1}_+)\) is complete. We also have the following density result: 

\begin{lemma} \label{lem-B-1} Let  \(1 \le p < +\infty\), \(\gamma \in \Rset\) and \(A \in C(\Rset^{d + 1}_+, \Rset^{d + 1})\). If \(1 - p < \gamma < 1\), then the space  \(C^{\infty}_{c}(\overline {\Rset^{d+1}_+})\) is dense in \(W^{1, p}_{A, \gamma}(\Rset^{d+1}_+)\). 
\end{lemma}

\begin{proof}
  \resetconstant
The proof of the completeness is standard. For the density of smooth maps, we first observe that if \(\chi \in C^\infty_c (\Rset^{d + 1})\), $0 \le \chi \le 1$,  and \(\chi = 1\) on \(B (0, 1)\) and if we define for each \(\lambda >0\) the function \(\chi_\lambda : \Rset^{d + 1} \to \Rset\) for \(x \in \Rset^d\) by \(\chi_{\lambda} (x) = \chi (x/\lambda)\),
then 
\(
 \nabla_A (U - \chi_\lambda U) = - U \nabla \chi_\lambda + (1 - \chi) \nabla_A U. 
\)
It follows that \begin{multline*}
\norm{U - \chi_{\lambda} U}_{W^{1, p}_{A, \gamma} (\Rset^d)}^p\\
\le \C \iint\limits_{\Rset^{d + 1}_+} 
\Bigl( \big(1 - \chi_\lambda (x, t) \big)^p \big(\abs{U (x, t)}^p + \abs{\nabla_A U (x, t)}^p\big) + 
  \abs{\nabla \chi_\lambda}^p \abs{U (x, t)}^p\Bigr) t^\gamma \dif x \dif t \to 0,
\end{multline*}
as \(\lambda \to \infty\), since \(\abs{\nabla \chi_\lambda} \le \C/\lambda\). Functions in \(W^{1, p}_{A, \gamma} (\Rset^{d + 1}_+)\) with bounded support are thus dense in \(W^{1, p}_{A, \gamma} (\Rset^{d + 1}_+)\).

Since \(1 - p < \gamma < 1\), any compactly supported can be approximated in \(W^{1, p}_{0, \gamma} (\Rset^{d + 1}_+)\)
by smooth functions with bounded support \citelist{\cite{Miller_1982}*{Lemma 2.4}\cite{Muckenhoupt_Weeden_1978}*{Lemma 8}\cite{Turesson_2000}*{Corollary 2.1.5}}
(the condition ensures that the weight \((x,t) \mapsto t^\gamma\) satisfies Muckenhoupt's \(A_p\) condition for \(p \in [1, +\infty)\) given by the assumptions) 
\citelist{\cite{Muckenhoupt_1972}\cite{Coifman_Fefferman_1974}}.
Since the function \(A\) is locally bounded on \(\Rset^{d + 1}\), this implies that any bounded supported function in \(W^{1, p}_{A, \gamma} (\Rset^{d + 1}_+)\) can be approximated in \(W^{1, p}_{0, \gamma} (\Rset^{d + 1}_+)\)
by smooth functions with bounded support and the conclusion follows by a diagonal argument.
\end{proof}

It is also standard to check that the space \(W^{s, p}_{A^\shortparallel} (\Rset^{d})\) is complete and thus is a Banach space. We also have the following density result: 

\begin{lemma}\label{lem-B-2} Let \(0< s < 1\) and \(p \ge 1\). If \(A^\shortparallel \in C^1(\Rset^d, \Rset^d)\), then the space \(C^{\infty}_{c} (\Rset^{d})\) is dense in \(W^{s, p}_{A^\shortparallel} (\Rset^{d})\). 
\end{lemma}

\begin{proof} 
  \resetconstant
First we observe that if \(u \in W^{s, p}_{A^\shortparallel}(\Rset^d)\) is arbitrary, \(\chi \in C^\infty_c(\Rset^d)\) is chosen with \(0 \le \chi \le 1\) and \(\chi = 1\) for \(x \in B (0, 1)\), and \(\chi_\lambda (x)  \defeq \chi (x/ \lambda)\),
then we have for each \(x, y \in \Rset^d\) and \(\lambda > 0\),
\begin{multline*}
    \bigl(
      1 - \chi_\lambda (y)
    \bigr)
    e^{i\,\intpot{A^\shortparallel} (x, y)}  u (y)
  - 
    \bigl(
      1 - \chi_{\lambda} (x)
    \bigr) 
    u (x)
    \\
=
  \bigl(1 - \tfrac{\chi_\lambda (x) + \chi_{\lambda} (y)}{2}\bigr) 
  \bigl(
      e^{i\,\intpot{A^\shortparallel} (x, y)} u (y)
    -  
      u (x)
  \bigr)
+ 
  \tfrac{\chi_{\lambda} (x) - \chi_\lambda (y)}{2} 
  \Bigl(
      e^{i\,\intpot{A^\shortparallel} (x, y)} u (y)
    +  
      u (x)
  \Bigr)
  ,
\end{multline*}
and for every \(y \in \Rset^d\) and \(\lambda > 0\),
\[
\int_{\Rset^d} \frac{\abs{\chi_\lambda (x) - \chi_\lambda (y)}^p}{\abs{y - x}^{d + sp}} \dif x
\le \frac{\C}{\lambda^{sp}}. 
\]
It follows that,  for every \(\lambda > 0\), 
\begin{multline*}
    |
      \chi_\lambda u - u
    |_{W^{s, p}_{A^\shortparallel} (\Rset^{d})}^p\\
  \le 
  \C 
  \iint\limits_{\Rset^d\times \Rset^d} 
    \Bigl(1 - \tfrac{\chi_\lambda (x) +  \chi_{\lambda} (y)}{2}\Bigr)^p \frac{\bigabs{
  e^{i\,\intpot{A^\shortparallel} (x, y)} u (y)
  -  u (x)}^p}{\abs{y - x}^{d + sp}}\dif y \dif x
    +
\C
\int_{\Rset^d}
\frac{\abs{u (y)}^p}{\lambda^{sp}} \dif y,
\end{multline*}
By Lebesgue's dominated convergence theorem  we deduce that 
\begin{align*}
  |\chi_\lambda u - u|_{W^{s, p}_{A^\shortparallel} (\Rset^d)} &\to 0&
  &\text{ and } &
  \Vert\chi_\lambda u - u\Vert_{L^p (\Rset^d)} &\to 0 &
  &\text{ as } \lambda \to + \infty. 
\end{align*}
Hence functions with compact support are dense in \(W^{s, p}_{A^\shortparallel} (\Rset^d)\).

We conclude by observing that since \(A^\shortparallel\) is locally bounded, any function 
in \(W^{s, p}_{A^\shortparallel} (\Rset^d)\) with compact support is in  \(W^{s, p}_{0} (\Rset^d)\),
such functions can be approximated in  \(W^{s, p}_{0} (\Rset^d)\) by functions with uniformly compact support;
since the function \(A^\shortparallel\) is locally bounded, the approximating sequence also converges in \(W^{s, p}_{A^\shortparallel} (\Rset^d)\). 
The conclusion then follows by a diagonal argument.
\end{proof}

As a consequence of Propositions~\ref{proposition_Bdd_dA_trace}, \ref{proposition_Bdd_dA_extension},  \ref{proposition_Bdd_dA_trace},  and \ref{proposition_Bdd_dA_extension}, and Lemmas~\ref{lem-B-1} and \ref{lem-B-2}, we obtain 
\begin{theorem} 
  \label{theorem_trace_halfplane_weight}
  Let \(d \ge 1\), \(0 < s < 1\) and \(1 \le p < +\infty\).  
  If 
  \(A \in C^1(\overline{\Rset^{d + 1}_+},  \Rset^{d + 1})\) and \(\norm{d A}_{L^\infty (\Rset^{d + 1}_+)} \le \beta\), then there exists a trace mapping 
\(\operatorname{Tr}: W^{1, p}_{A, 1 - (1-s)p} (\Rset^{d+1}_+, \Cset) \to W^{s, p}_{A^{\shortparallel}}(\Rset^d)\)
such that  for some positive constant \(C_{d, s, p}\) depending only on \(d\), \(s\) and \(p\),
if  \(U \in C^1_c (\Rset^{d + 1}_+, \Cset)\), then \(\operatorname{Tr} U = U (\cdot, 0)\)
and for every \(U \in W^{1, p}_{A, 1 - (1 - s)p} (\Rset^{d + 1}_+)\), if \(u \defeq \operatorname{Tr} U\),
\begin{equation*} \label{extension-dA-constant-111-1}
  \seminorm{u}_{W^{s, p}_{A^{\shortparallel}}(\Rset^d)}^p 
  + \beta^\frac{sp}{2}
  \norm{u}_{L^p(\Rset^d)}^p  
  \le 
  C_{d, s, p}  
  \iint\limits_{\Rset^{d + 1}_+}
  \frac{\abs{\nabla_A U (x, t)}^p + \beta^\frac{p}{2} \abs{U (z, t)}^p}
  {t^{1 - (1 - s)p}} \dif x \dif t,
\end{equation*}
and there exists a linear continuous mapping 
\(
\operatorname{Ext}: W^{s, p}_{A^{\shortparallel}}(\Rset^d, \Cset) \to W^{1, p}_{A, 1 - (1 - s)p} (\Rset^{d+1}_+)
\) 
such that  \(\operatorname{Tr} \circ \operatorname{Ext}: W^{s, p}_{A^{\shortparallel}}(\Rset^d)  \to W^{s, p}_{A^{\shortparallel}}(\Rset^d) \) is the identity and such that for some positive constant \(C_{d, s, p}\) depending only on \(d\), \(s\) and \(p\), we have for each \(u \in W^{s, p}_{A^\shortparallel} (\Rset^d)\), if \(U \defeq \operatorname{Ext} u\),
\begin{equation*} 
  \label{extension-dA-constant-111-2}
   \iint\limits_{\Rset^{d + 1}_+}
   \frac{\abs{\nabla_A U (x, t)}^p + \beta^\frac{p}{2} \abs{U (z, t)}^p}
      {t^{1 - (1 - s)p}} \dif x \dif t\\
  \le 
  C_{d, s, p}
  \Bigl( \seminorm{u}_{W^{s, p}_{A^{\shortparallel}}(\Rset^d)}^p 
  + \beta^\frac{sp}{2}\norm{u}_{L^p(\Rset^d)}^p  \Bigr).
\end{equation*}
\end{theorem}

\section{Constant magnetic field on the half-space}

We begin with an improvement of  \cref{proposition_Bdd_dA_trace} in the case where the magnetic field \(dA\) is constant. 

\begin{proposition}
  \label{proposition_trace_estimate-1}
  Let \(d \ge 1\), \(0 < s < 1\), and \(1 \le p < + \infty\).  
  There exists a constant \(C_{d, s, p}> 0\) such that if
  \(A \in C^1 (\Rset_+^{d + 1}, \Rset^{d + 1})\) with constant  \(dA\), \(U \in C^1_c(\overline{\Rset^{d+1}_+})\) and \(u = U (\cdot, 0)\), then 
  \begin{equation*}
    \seminorm{u}_{W^{s, p}_{A^{\shortparallel}} (\Rset^d)}^p 
    + \norm{dA}^{\frac{sp}{2}} \int_{\Rset^d} \abs{U (x, 0)}^p \dif x
    \le 
    C_{d, s, p}
    \iint\limits_{\Rset^{d + 1}_+}
    \frac{\bigabs{\nabla_A U (z, t)}^p}{t^{1 + (1 - s)p}} 
    \dif z \dif t.
  \end{equation*}
\end{proposition}

The first ingredient of the proof  \cref{proposition_trace_estimate-1} is the following lemma:

\begin{lemma}
  \label{lemma_integraltoboundary} Let \(d  \ge 1\), \(0< s < 1\) and \( 1\le p <  +\infty\).
  For every \(\lambda > 0\), there exists a positive constant \(C_{d, s, p, \lambda}\) such that,
  for every  \(A \in C^1 (\overline{\Rset_+^{d + 1}}, \bigwedge^1 \Rset^{d + 1})\) with constant  \(dA\), we have, for \(U \in C^1_c(\overline{\Rset^{d+1}_+}, \Cset)\),
  \begin{equation*}
    \iint\limits_{\Rset^d \times \Rset^d}
    \frac{ \bigabs{
        e^{
          i\,(\intpot{A^{\shortparallel}}(x, y)
          + 
          \lambda  |y-x| d A [y-x, e_{d+1}] )
        } 
        U (y, 0) 
        - 
        U (x, 0)
      }^p}{\abs{y - x}^{d + s p}}
    \dif x \dif y  \le 
    C_{d, s, p, \lambda}
    \iint\limits_{\Rset^{d + 1}_+}
    \frac{\bigabs{\nabla_A U (z, t)}^p}{t^{1 - (1 - s)p}}
    \dif z \dif t. 
  \end{equation*}
\end{lemma}

\begin{proof}
  \resetconstant
  For \(x, y \in \Rset^d\), we identify \(x = (x, 0)\), \(y = (y, 0)\), and we set 
  \(Z  = (\frac{x+ y}{2}, 2 \lambda \abs{y - x})\). 
  Since \(dA\) is constant,  by \cref{lem-Magnetic}, we have
  \begin{equation*}
    \intpot{A} (x, y)
    +
    \intpot{A} ( y,  Z) 
    +
    \intpot{A} (Z, x)
    = \frac{1}{2} d A [(y -x, 0), (\tfrac{x - y}{2}, -2 \lambda \abs{x - y})]
    =  - \lambda \nu_A (y - x),
  \end{equation*}
  where the function \(\nu_A : \Rset^d \to \Rset\) is defined 
  for each \(z \in \Rset^d\) by \(\nu_A (z) =  \abs{z}\, dA [z, e_{d+1}]$. This implies 
  by the triangle inequality
  \begin{equation}
    \label{eq_ohG7woo0pu}
    \Bigabs{
      e^{
        i\, (\intpot{A}
        (x, y)
        + 
        \lambda \nu_A (y - x) ) 
      } 
      U (y) 
      - 
      U (x) }
    \le
    \Big|
    e^{i\,\intpot{A} ( x, Z)}  
    U (Z) - U (x) \Big|
    +
    \Big|e^{i\,\intpot{A} (y, Z)}   U (y) - U (Z)\Big|. 
  \end{equation}
  Using \eqref{mvl-A}, we deduce from \eqref{eq_ohG7woo0pu} that 
  \begin{multline*}
    \Big|
    e^{
      i\, (\intpot{A}
      (x, y)
      + 
      \lambda \nu_A (y - x) ) 
    } 
    U (y) 
    - 
    U (x) \Big|
    \\
    \le    
    \sqrt{1 + \lambda^2} \,\abs{y - x}
    \biggl(
    \int_0^1 
    \bigabs{\nabla_A U \big((1-\tfrac{t}{2}) x + \tfrac{t}{2} y, 2 \lambda \abs{y - x} \big)}
    \dif t \\[-.5em]
    +      
    \int_0^1 
    \bigabs{\nabla_A U  \big((1-\tfrac{t}{2}) y + \tfrac{t}{2} x, 2 \lambda \abs{y - x} \big)}
    \dif t\Biggr). 
  \end{multline*}
  We then have, by Minkowski's inequality,
  \begin{multline}\label{lem3.2-0}
    \iint\limits_{\Rset^d \times \Rset^d}
    \frac{  \bigabs{
        e^{
          i\, (\intpot{A}
          (x, y)
          + 
          \lambda \nu_A (y - x) ) 
        } 
        U (y) 
        - 
        U (x) }^p}{\abs{y - x}^{d + s p}} \dif x \dif y \\
    \le 
    \C 
    (1 + \lambda)^p
    \Biggl( \int_0^1 \Biggl( \; \iint\limits_{\Rset^d \times \Rset^d}  
    \frac{ 
      \bigabs{\nabla_A U  \big((1-\tfrac{t}{2}) x + \tfrac{t}{2} y, 2 \lambda \abs{y - x} \big)}^p}
    {|y - x|^{d - (1 - s) p}}  \dif y \dif x 
    \Biggr)^{1/p} \\
    + 
    \Biggl( \;\iint\limits_{\Rset^d \times \Rset^d}  
    \frac{
      \bigabs{\nabla_A U  \big((1-\tfrac{t}{2}) y + \tfrac{t}{2} x, 2 \lambda \abs{y - x} \big)}^p}
    {|y - x|^{d - (1 - s) p}}  \dif y \dif x 
    \Biggr)^\frac{1}{p}
    \dif t \Biggr)^{p}. 
  \end{multline}
  Performing the change of variable \(z = (1 - \frac{t}{2}) x + \tfrac{t}{2} y\) and \(v = 2 \lambda t (x - y)\), we obtain 
  \begin{equation*}
    \iint\limits_{\Rset^d \times \Rset^d}      
    \frac
    { \bigabs{\nabla_A U  \big((1-\tfrac{t}{2}) x + \tfrac{t}{2} y, 2 \lambda \abs{y - x} \big)}^p}
    {\abs{y - x}^{d - (1 - s) p}}
    \dif y \dif x    =
    \frac{1}{(2 \lambda t)^{(1 - s) p}} \iint\limits_{\Rset^d \times \Rset^d}      
    \frac
    {\bigabs{\nabla_A U (z, \abs{v})}^p}
    {\abs{v}^{d - (1 - s) p}}
    \dif z 
    \dif v. 
  \end{equation*}
  This yields by using spherical coordinates, for every \(t \in (0, 1)\),
  \begin{equation}\label{lem3.2-1}
    \iint\limits_{\Rset^d \times \Rset^d}    
    \! \! \!     
    \frac
    { \bigabs{\nabla_A U  \big((1-\tfrac{t}{2}) x + \tfrac{t}{2} y, 2 \lambda \abs{y - x} \big)}^p}
    {\abs{y - x}^{d - (1 - s) p}}
    \dif x \dif y  
    \le 
    \frac{\Cl{cst_Hoshoobe9yei4lahw}}{(\lambda t)^{p (1 - s)}} \! \! \!     
    \iint\limits_{\Rset^d \times (0, + \infty)} \! \! \!     
    \frac
    {\bigabs{\nabla_A U (z, r)}^p}
    {r^{1 - (1 - s) p}}
    \dif z \dif r.    
  \end{equation}
  Similarly, we have 
  \begin{equation}
    \label{lem3.2-2}
    \iint\limits_{\Rset^d \times \Rset^d}    
    \! \! \!     
    \frac
    { \bigabs{\nabla_A U  \big((1-\tfrac{t}{2}) y + \tfrac{t}{2} x, 2 \lambda \abs{y - x} \big)}^p}
    {\abs{y - x}^{d - (1 - s) p}}
    \dif x \dif y  
    \le 
    \frac{\Cr{cst_Hoshoobe9yei4lahw}}{(\lambda t)^{p (1 - s)}}
    \! \! \!     
    \iint\limits_{\Rset^d \times (0, + \infty)}    \! \! \!        
    \frac
    {\bigabs{\nabla_A U (z, r)}^p}
    {r^{1 - (1 - s) p}}
    \dif z \dif r.    
  \end{equation}
  Combining \eqref{lem3.2-0} with \eqref{lem3.2-1}, and \eqref{lem3.2-2} yields 
  \begin{multline}
    \label{eq_aeKohPo8iX}
    \iint\limits_{\Rset^d \times \Rset^d}
    \frac{ \bigabs{
        e^{
          i\,(\intpot{A^{\shortparallel}}(x, y)
          + 
          \lambda \nu_A (y-x))
        } 
        U (y, 0) 
        - 
        U (x, 0)
      }^p}{\abs{y - x}^{d + s p}}
    \dif x \dif y
    \\
    \le
    \frac{\C (1 + \lambda)^{p}}{\lambda^p} 
    \left(\int_0^1
    \frac{1}{t^{1 - s}} \dif t\right)^p
    \int_{\Rset^{d + 1}_+}
    \bigabs{\nabla_A U(z, r)}^p \dif  z \dif r
  \end{multline}
  and the conclusion follows. 
\end{proof}

The second tool is the following fractional magnetic Poincar\'e inequality.

\begin{lemma}
  \label{lem_Poincare}
  Let \(d  \ge 1\), \(0 < s < 1\) and \(1 \le p < +\infty\) and \(A^\shortparallel \in C^1 (\Rset^{d}, \bigwedge^1 \Rset^{d})\) with constant $dA$. 
  There exists a constant \(C_{d, s, p}> 0\) such that, for \(u \in C^1_{c}(\Rset^d)\), 
  \[
  \int_{\Rset^d} \abs{u}^p
  \le 
  \frac{C_{d, s, p}}{\norm{dA^\shortparallel}^\frac{sp}{2}}
  \iint\limits_{\Rset^d \times \Rset^d}
  \frac
  {\bigabs{e^{i\,\intpot{A^\shortparallel} (x, y)} u (y) - u (x)}^p}
  {\abs{y - x}^{d + sp}}
  \dif x
  \dif y.
  \]
\end{lemma}

The counterpart of \cref{lem_Poincare} in \(W^{1, 2}_A (\Rset^{d})\) is known \citelist{\cite{Avron_Herbst_Simon_1978}*{Theorem 2.9}\cite{Esteban_Lions_1989}*{Proposition 2.2}} and related to the positiveness of the first eigenvalue of the magnetic Laplacian \(-\Delta_A\), which corresponds to the first Landau level.

\begin{proof}[Proof of \cref{lem_Poincare}]
  \resetconstant
  Let \(J : \Rset^d \to \Rset^d\) be a linear isometry. 
  We observe that by setting \(z = x + h\) and \(k = Jh\), we have
  \begin{multline}
    \label{eq_frac_comm_2}
    \iint
    \limits_{\Rset^d \times \Rset^d}
    \frac
    {\bigabs{e^{i\,\intpot{A} (x +h, x + h + Jh)} u (x + h + Jh) - u (x + h)}^p}
    {\abs{h}^{d + s p}}
    \dif x
    \dif h
    \\[-1em]
    =
    \iint
    \limits_{\Rset^d \times \Rset^d}
    \frac
    {\bigabs{e^{i\,\intpot{A} (z, z + k)} u (z + k) - u (z)}^p}
    {\abs{k}^{d + s p}}
    \dif z
    \dif k.
  \end{multline}
  Similarly, by setting \(z = x + h + Jh\) and \(k = -h\), we have 
  \begin{multline}
    \label{eq_frac_comm_3}
    \iint
    \limits_{\Rset^d \times \Rset^d}
    \frac
    {
      \bigabs{e^{i\,\intpot{A} (x + h + Jh, x + Jh)} 
        u (x + Jh) 
        - 
        u (x + h + Jh)}^p
    }
    {\abs{h}^{d + s p}}
    \dif x
    \dif h
    \\[-1em]
    =
    \iint
    \limits_{\Rset^d \times \Rset^d}
    \frac
    {\bigabs{e^{i\,\intpot{A} (z, z + k)} u (z + k) - u (z)}^p}
    {\abs{k}^{d + s p}}
    \dif z
    \dif k.
  \end{multline}
  Finally, we have by setting \(z = x + Jh\) and \(k = -h\),
  \begin{multline}
    \label{eq_frac_comm_4}
    \iint
    \limits_{\Rset^d \times \Rset^d}
    \frac
    {
      \bigabs{e^{i\,\intpot{A} (x + Jh, x)} 
        u (x) 
        - 
        u (x + Jh)}^p
    }
    {\abs{h}^{d + s p}}
    \dif x
    \dif h
    \\[-1em]
    =
    \iint
    \limits_{\Rset^d \times \Rset^d}
    \frac
    {\bigabs{e^{i\,\intpot{A} (z, z + k)} u (z + k) - u (z)}^p}
    {\abs{k}^{d + s p}}
    \dif z
    \dif k.
  \end{multline}
  We compute now, since \(dA\) is constant, by Stokes formula 
  \begin{multline*}
    \intpot{A} (x, x + h)
    + \intpot{A} (x + h, x + h + Jh)
    + \intpot{A} (x + Jh, x + h + Jh)
    +\intpot{A} (x + Jh, h)
    = d A [h, Jh], 
  \end{multline*}
  so that, by the triangle inequality,
  \begin{multline}    \bigabs{e^{d A [h, Jh]} - 1}\abs{u (x)} \\
    \le 
    \bigabs{e^{i\intpot{A} (x, x + h)} u (x + h) - u (x)}
    + 
    \bigabs{e^{i\intpot{A} (x + h, x + h + Jh)} u (x + h + Jh) - u (x + h)}
    \\
    +
    \bigabs{e^{i\intpot{A} (x + Jh, x + h + Jh)} 
      u (x + h) 
      - 
      u (x + h + Jh)}
    +\bigabs{e^{i\,\intpot{A} (x + Jh, h)} 
      u (x) 
      - 
      u (x + Jh)}
    .
  \end{multline}
  Therefore, we have by H\"older's inequality, in view of 
  \eqref{eq_frac_comm_2}, \eqref{eq_frac_comm_3} and \eqref{eq_frac_comm_4}, 
  \begin{equation*}
    \iint
    \limits_{\Rset^d \times \Rset^d}
    \frac
    {
      \bigabs{e^{d A [h, Jh]} - 1}^p
      \abs{u (x)}^p
    }
    {\abs{h}^{d + s p}}
    \dif x
    \dif h
    \le
    4^{p - 1}
    \iint
    \limits_{\Rset^d \times \Rset^d}
    \frac
    {\bigabs{e^{i\,\intpot{A} (x, x + h)} u (x) - u (x + h)}^p}
    {\abs{h}^{d + s p}}
    \dif x
    \dif h.   
  \end{equation*}
  Integrating with respect to \(J\) over the group \(SO_d\) of rotations of \(\Rset^d\), we obtain  \begin{equation*}
    \norm{dA}^{sp/2} \int_{\Rset^d}      \abs{u (x)}^p
    \dif x
    \le
    \C   \iint
    \limits_{\Rset^d \times \Rset^d}
    \frac
    {\bigabs{e^{i\,\intpot{A} (x, x + h)} u (x) - u (x + h)}^p}
    {\abs{h}^{d + s p}}
    \dif x
    \dif h.   
  \end{equation*} 
  The proof is complete. 
\end{proof}

We are ready to give 

\begin{proof}[Proof of \cref{proposition_trace_estimate-1}]
  Applying \cref{lemma_integraltoboundary} with \(\lambda = 1\) and \(\lambda = 2\), and using the triangle inequality, we obtain 
  \begin{equation}
    \label{eq_aey7hahv0R}
    \iint\limits_{\Rset^d \times \Rset^d}
    \frac{\abs{
        e^{
          i  |y - x| dA [y-x, e_{d+1}] 
        } 
        - 1}^p
      \abs{U (y, 0)}^p}
    {\abs{y - x}^{d + s p}}
    \dif x 
    \dif y
    \le
    \C 
    \iint\limits_{\Rset^{d + 1}_+}
    \frac{\bigabs{\nabla_A U (z, t)}^p}{t^{1 - (1 - s) p} } \dif z\dif t. 
  \end{equation}
Since 
\begin{align*}
\abs{
        e^{
          i\,\intpot{A^{\shortparallel}}
          ( x, y )
        } 
        U (y, 0) 
        - 
        U (x, 0)
      }^p \le &  2^{p-1}  \bigabs{
        e^{
          i\,(\intpot{A^{\shortparallel}}
          ( x, y )
          + |y - x| dA [y-x, e_{d+1}]) 
        } 
        U (y, 0) 
        - 
        U (x, 0)
      }^p \\[6pt]
     & \quad  +  2^{p-1} \bigabs{
        e^{
          i  |y - x| dA [y-x, e_{d+1}] 
        } 
        - 1}^p
      \abs{U (y, 0)}^p,  
      \end{align*}
it follows from \eqref{eq_aey7hahv0R} and Lemma~\ref{lemma_integraltoboundary} with $\lambda =1$ that 
  \begin{equation}
    \label{eq_phuop5Nong}
    \iint\limits_{\Rset^d \times \Rset^d}
    \frac{\abs{
        e^{
          i\,\intpot{A^{\shortparallel}}
          ( x, y )
        } 
        U (y, 0) 
        - 
        U (x, 0)
      }^p}{\abs{y - x}^{d + s p}}
    \dif x \dif y\\
        \le 
    \C
    \iint\limits_{\Rset^{d + 1}_+}
    \frac{\bigabs{\nabla_A U (z, t)}^p}{t^{1 - (1 - s) p} } \dif z\dif t
    .
  \end{equation}
From  \cref{lem_Poincare},
  we have 
  \begin{equation}
    \label{eq_ga1aiPu7tee7Iqu0}
    \norm{dA^\shortparallel}^\frac{sp}{2}
    \int_{\Rset^d} \abs{U (\cdot, 0)}
    \le \C
    \iint\limits_{\Rset^{d + 1}_+}
    \frac{\bigabs{\nabla_A U (z, t)}^p}{t^{1 - (1 - s) p} } \dif z\dif t.
  \end{equation}
Combining  \eqref{eq_phuop5Nong} and \eqref{eq_ga1aiPu7tee7Iqu0} yields the conclusion. 
\end{proof}

Using Proposition~\ref{proposition_trace_estimate-1} and Theorem~\ref{theorem_trace_halfplane_weight}, we obtain the following result which implies Theorem~\ref{thm-3-intro}  in the introduction.  

\begin{theorem} \label{thm-3}
  Let \(d \ge 1\), \(0 < s < 1\) and \(1 \le p < +\infty\).  Assume that 
  \(A \in C^1(\overline{\Rset^{d + 1}_+},  \Rset^{d + 1})\) and $dA$ is constant. 
  Then, with  \(u \defeq \operatorname{Tr} U\),
\begin{equation*} \label{extension-dA-constant-111-1}
 \seminorm{u}_{W^{s, p}_{A^{\shortparallel}}(\Rset^d)}^p 
  + \beta^\frac{sp}{2}
  \norm{u}_{L^p(\Rset^d)}^p 
  \le 
  C_{d, s, p}  
  \iint\limits_{\Rset^{d + 1}_+}
  \frac{\abs{\nabla_A U (x, t)}^p + \|dA \|^\frac{p}{2} \abs{U (z, t)}^p}
  {t^{1 - (1 - s)p}} \dif x \dif t,
\end{equation*}
and  with \(U \defeq \operatorname{Ext} u\),
\begin{equation*} 
  \label{extension-dA-constant-111-2}
   \iint\limits_{\Rset^{d + 1}_+}
   \frac{\abs{\nabla_A U (x, t)}^p + \|dA \|^\frac{p}{2} \abs{U (z, t)}^p}
      {t^{1 - (1 - s)p}} \dif x \dif t\\
  \le 
  C_{d, s, p}
  \Bigl( \seminorm{u}_{W^{s, p}_{A^{\shortparallel}}(\Rset^d)}^p 
  + \beta^\frac{sp}{2}\norm{u}_{L^p(\Rset^d)}^p  \Bigr), 
\end{equation*}
for some positive constant \(C_{d, s, p}\) depending only on \(d\), \(s\) and \(p\), 
\end{theorem}

\section{Trace and extension on domains}\label{sect-domain}

In this section, we consider the trace problem on a domain \(\Omega\)  of class $C^1$ with estimates depending only on the magnetic field \(dA\).  We first develop the tools to work with a magnetic derivative on the  boundary \(\partial \Omega\) via local charts.  Let  \(W \subset \Rset^d\) be an open set, \(\psi \in C^1 (W, \Rset^d)\),  and \(A : \psi (W) \to \bigwedge^1 \Rset^d\).  The pull-back \(\psi^*\! A\) of \(A\) by \(\psi\) is defined for each \(x \in W\) and \(v \in \Rset^n\) by 
\begin{equation}
  \label{eq_def_pull_back}
    (\psi^*\! A) (x)
  \defeq 
  D \psi (x)^* A (\psi (x)),
\end{equation}
where \(D \psi (x)^*\) is the adjoint of \(D \psi (x)\). We first recall the following elementary result whose proof follows from the chain rule 
(see,  e.g.,  \cite{Willem_2013}*{Proposition 6.1.11}) and the definition of the pull-back \eqref{eq_def_pull_back}. 

\begin{lemma}
\label{lemma_covariant_derivative_pullback}
Let \(d \ge 1\),  \(V \subset \Rset^d\) be open, bounded, and let  \(\psi \in C^1 (V, \Rset^{d + 1})\) be a diffeomorphism on its image up to the boundary. 
Then  \(U \in W^{1, p}_A (\psi(V))\) if and only if \(U \compose \psi \in W^{1, p}_{\psi^*\! A} (V)\).
Moreover, for almost every \(x \in V\),
\begin{equation}\label{lem-TO-p1}
\nabla_{\psi^*\! A} (U \compose \psi) (x) = D \psi (x)^* [(\nabla_A U)(\psi (x))]. 
\end{equation}
Consequently,  
\begin{equation}\label{lem-TO-p2}
  \frac{\abs{\nabla_A U\, (\psi(x))}}{\norm{D \psi^{-1}}_{L^\infty}} 
  \le 
  \abs{\nabla_{\psi^*\! A} (U \compose \psi) (x)}
  \le 
  \norm{D \psi}_{L^\infty} \abs{\nabla_A U\, (\psi (x))}.
\end{equation}
\end{lemma}

We next derive a similar result on the boundary when $\psi: W \mapsto \partial \Omega$ for the fractional magnetic Gagliardo--Sobolev seminorm. 
Although the potential \(\intpot{\Bar{A}}\) defined in \eqref{eq_def_intpot} does not make sense in general for each \(x, y \in \partial \Omega\) if \(\Omega\) is not convex for \(A^\shortparallel \in C (\bigwedge^1 T^* \partial \Omega)\),
when \(\partial \Omega\) is compact smooth manifold, then \(\partial \Omega\) has a positive injectivity radius \(\operatorname{inj}_{\partial \Omega}\) and if \(x, y \in \partial \Omega\) and \(d_{\partial \Omega} (y, x) \le \operatorname{inj}_{\partial \Omega}\), then there exists a unique minimizing geodesic \(\gamma : [0, 1] \to \partial \Omega\) such that \(\gamma (0) = x\) and \(\gamma (1) = y\). We then define for such $x$ and $y$ the quantity 
\begin{equation} \label{def-IA-bdry}
\intpot[\partial \Omega]{A^\shortparallel} (x, y)
 \defeq 
  \int_0^1 
  A^\shortparallel \bigl(\gamma (t) \bigr) \cdot 
      \gamma' (t) 
      \dif t.
\end{equation}

We have 

\begin{lemma}
  \label{lemma_fract_domain_boundary_local}
  Let $d \ge 1$, \(W \subset \Rset^{d}\) be open, bounded,  and  let \(\psi : W \to \Rset^{d+1}\) be a diffeomorphism  up to the boundary
to its image as a subset of the manifold \(\partial \Omega\). If \(W\) and \(\psi(W)\) are geodesically convex, 
then for every \(0< s < 1\) and \(1 \le p <  +\infty\), there exists a positive constant \(C\) such that for every \(A^\shortparallel \in C^1 (\bigwedge^1 T^* \partial \Omega)\) and every measurable function \(u : \psi (W) \to \Cset\),
\begin{multline*}  
C^{-1}   \iint\limits_{\psi(W) \times \psi(W)} 
  \frac{\bigabs{e^{i\,\intpot[\partial \Omega]{A^{\shortparallel}} (x, y)} u (y) - u (x)}^p}{d_{\partial \Omega} (y,x)^{d + s p}} \dif y \dif x\\
  \le 
 \iint\limits_{W \times W} 
  \frac{\bigabs{e^{i\,\intpot{\psi^*\! A^{\shortparallel}} (x, y)} u (\psi (y)) - u (\psi(x))}^p}{\abs{y - x}^{d + s p}} \dif y \dif x 
  +  \min (\norm{dA^{\shortparallel}}_{L^\infty}^p, \norm{dA^{\shortparallel}}_{L^\infty}^\frac{sp}{3}) \int_{W} \abs{u \compose \psi}^p \, dx \end{multline*}
and
\begin{multline*}
  \iint\limits_{W \times W} 
  \frac{\bigabs{e^{i\,\intpot{\psi^*\! A^{\shortparallel}} (x, y)} u (\psi (y)) - u (\psi(x))}^p}{\abs{y - x}^{d + s p}} \dif y \dif x\\
  \le 
  C \Biggl(\;\;
  \iint\limits_{\psi(W) \times \psi(W)} 
  \frac{\bigabs{e^{i\,\intpot[\partial \Omega]{A^{\shortparallel}} (x, y)} u (y) - u (x)}^p}{d_{\partial \Omega} (y,x)^{d + s p}} \dif y \dif x
  + \min \bigl\{ \norm{dA^{\shortparallel}}_{L^\infty}^p, \norm{dA^{\shortparallel}}_{L^\infty}^\frac{sp}{3}
  \bigr\} \int_{\psi (W)} \abs{u}^p \, dx\Biggr).
\end{multline*}
\end{lemma}

In the proof of Lemma~\ref{lemma_fract_domain_boundary_local}, we use the following result which  relates the potential \(\intpot[\partial \Omega]{A}\) to the potential \(\intpot{\psi^*\!A}\) via local charts.

\begin{lemma}
  \label{lemma_transport_intpot}
  Let \(d \ge 1\), let \(W \subset \Rset^{d}\) be open and bounded,  and let \(\psi : W \to \partial \Omega\) be a diffeomorphism up to the boundary to its image. 
  If \(W\) and \(\psi(W)\) are geodesically convex, 
then there exists a positive constant $C$ such that for every \(x, y \in W\), we have 
\begin{equation*}
    \bigabs{
        \intpot{\psi^*\! A^{\shortparallel}} (x, y) 
      - 
        \intpot[\partial \Omega]{A^{\shortparallel}} (\psi (x), \psi (y))
      }
  \le
    C
    \norm{d A^{\shortparallel}}_{L^\infty (\psi (W))}
    \abs{y - x}^3.
\end{equation*}
\end{lemma}

\begin{proof}
\resetconstant%
For every \(x, y \in W\), there exists a unique minimizing geodesic \(\gamma : [0, 1] \to \psi (W)\) such that \(\gamma (0) = \psi (x)\) and \(\gamma (1) = \psi (y)\). 
Since \(\gamma\) is a geodesic, the function \(\Tilde{\gamma} = \psi^{-1} \compose \gamma\) satisfies the equation
\begin{equation*}
  \label{eq_Aithi3ooKoopo}
  \Tilde{\gamma}'' (t) 
  = 
  \Gamma (\Tilde{\gamma} (t))[\Tilde{\gamma}' (t), \Tilde{\gamma}' (t)],
\end{equation*}
where for every \(z \in \psi (W)\), \(\Gamma (z)\) is a symmetric bilinear mapping (see,  e.g.,  \cite[Chapter 3]{Carmo}).
There exists thus a constant \(\Cl{cst_ooH3tahqu9aeZ}\) such that for every \(t \in [0, 1]\),
\begin{equation*}
  \abs{\Tilde{\gamma}'' (t)} \le  \Cr{cst_ooH3tahqu9aeZ} \abs{y - x}^2.
\end{equation*}
Since \(\Tilde{\gamma} (0) = x\) and \(\Tilde{\gamma} (1) = y\), we deduce that for every \(t \in [0, 1]\) we have 
\[
    \bigabs{(1 - t) x + t y - \Tilde{\gamma} (t)}
  \le 
     \C
      \, 
      t (1 - t)
      \,
      \abs{y - x}^2.
\]
We have then by the Stokes theorem
\begin{multline*}
\intpot{\psi^*\! A^{\shortparallel}} (y, x) 
- 
\intpot[\partial \Omega]{A^{\shortparallel}} (\psi (y), \psi (x))
\\
=
\int_{[0, 1]^2} d A \bigl((1 -s) ((1 - t) x + t y) + s \Tilde{\gamma} (t)\bigr) [(1 -t)x + t y - \Tilde{\gamma} (t), (1 - s) \Tilde{\gamma}' (t) + s (y - x)] \dif t \dif s,
\end{multline*}
and therefore,
\[
\bigabs{
\intpot{\psi^*\! A^{\shortparallel}} (y, x) 
- 
\intpot[\partial \Omega]{A^{\shortparallel}} (\psi (y), \psi (x))
}
\le 
\C
\norm{dA^{\shortparallel}}_{L^\infty(\psi (W))}
\abs{y - x}^3.\qedhere
\]
\end{proof}

We are ready to give the

\begin{proof}[Proof of \cref{lemma_fract_domain_boundary_local}]
  \resetconstant
By the change of variable formula, we have 
\begin{multline*}
  \iint\limits_{\psi(W) \times \psi(W)} 
    \frac
      {\bigabs{e^{i\,\intpot[\partial \Omega]{A^{\shortparallel}} (x, y)} u (y) - u (x)}^p}
      {d_{\partial \Omega} (y,x)^{d + s p}} \dif y \dif x\\[-.8em]
  =
    \iint\limits_{W \times W} 
      \frac
        {\bigabs{e^{i\,\intpot[\partial \Omega]{A^{\shortparallel}} (\psi(x), \psi(y))} u (\psi (y)) - u (\psi(x))}^p}
        {d_{\partial \Omega} (\psi(y), \psi(x))^{d + s p}} 
        \operatorname{Jac} \psi (y) \operatorname{Jac} \psi (x) 
        \dif y \dif x
        .
\end{multline*}
We thus obtain 
\begin{multline*}
  \iint\limits_{\psi(W) \times \psi(W)} 
  \frac
  {\bigabs{e^{i\,\intpot[\partial \Omega]{A^{\shortparallel}} (x, y)} u (y) - u (x)}^p}
  {d_{\partial \Omega} (y,x)^{d + s p}} \dif y \dif x\\[-.8em]
  \le 
  \C \Biggl(\;
  \iint\limits_{W \times W} 
  \frac
  {\bigabs{e^{i\,\intpot{\psi^*\! A^{\shortparallel}} (x, y)} u (\psi (y)) - u (\psi(x))}^p}
    {\abs{y - x}^{d + s p}} 
  \dif y \dif x\\
  +
    \iint\limits_{W \times W} 
  \frac
  {\Bigabs{e^{i\,\intpot[\partial \Omega]{A^{\shortparallel}}(\psi(x), \psi(y))} - e^{i\,\intpot{\psi^*\! A^{\shortparallel}} (x, y)}}^p \abs{u (\psi (y))}^p}
    {\abs{y - x}^{d + s p}} 
  \dif y \dif x
  \Biggr).
\end{multline*}
Since,  by \cref{lemma_transport_intpot},  for every \(x, y \in W\),
\[
  \frac{\Bigabs{e^{i\,\intpot[\partial \Omega]{A^{\shortparallel}}(\psi(x), \psi(y))} - e^{i\,\intpot{\psi^*\! A^{\shortparallel}} (x, y)}}^p}{\abs{y - x}^{d + sp}}
  \le \frac{\C \min \{\norm{d A^{\shortparallel}}_{L^\infty (\psi (W))}^p \abs{y - x}^{3p}, 1\}}{\abs{y - x}^{d + sp}},
\]
the first estimate then  follows from the facts
\[
\int_{\Rset^{d}} \frac{\min \bigl\{\norm{d A^{\shortparallel}}_{L^\infty(\psi (W))}^p \abs{z}^{3p}, 1\bigr\}}{\abs{z}^{d + sp}} \dif z 
\le \C\, \norm{d A^{\shortparallel}}_{L^\infty(\psi (W))}^\frac{sp}{3}
\]
and 
\[
\int_{B (0, {\operatorname{diam} (W) )}} \frac{\norm{d A^{\shortparallel}}_{L^\infty(\psi (W))}^p \abs{z}^{3 p}}{\abs{z}^{d + sp}} \dif z 
\le \C\, \norm{d A^{\shortparallel}}_{L^\infty(\psi (W))}^{p}.
\]

The proof of the second estimate follows similarly.
\end{proof}

We have the following estimate on the traces of magnetic Sobolev spaces.

\begin{proposition}
\label{proposition_Bdd_dA_trace_domain}
Let \(d \ge 1\), let \(1 \le p < +\infty\) and let \(\Omega \subset \Rset^{d + 1}\) be a bounded domain of class $C^1$.
There exists a constant \(C_{\Omega, p}> 0\) such that if \(A \in C^1 (\Bar{\Omega}, \bigwedge^1 \Rset^{d})\)
and \(\norm{dA}_{L^\infty (\Omega)} \le \beta\) and \(U \in W^{1, p}_{A} (\Omega)\), then 
\(u \defeq \operatorname{Tr} U\) satisfies the estimate
\begin{equation*}
  \iint\limits_{\substack{(x, y) \in \partial \Omega \times \partial \Omega\\
      d_{\partial \Omega} (y, x) \le \inj_{\partial \Omega}}}
    \frac{\bigabs{
      e^{
          i\,\intpot[\partial \Omega]{A^\shortparallel} (x, y) 
        } 
      u (y) 
    - 
      u (x)
      }^p}{\abs{y - x}^{d + p - 1}}
    \dif x \dif y
  \le
  C_{\Omega, p}
          \int_{\Omega}
          \bigabs{\nabla_A U }^p + ( 1 + \beta^\frac{p}{2}) \abs{U}^p.
\end{equation*}
Here for $z \in \partial \Omega$, $A^{\shortparallel}(z) \defeq A (z) - (  A(z) \cdot \nu (z)) \nu (z)$ where $\nu (z)$ denotes a unit normal vector of $\partial \Omega$ at $z$. 
\end{proposition}


\begin{proof}%
[Proof of \cref{proposition_Bdd_dA_trace_domain}]%
\resetconstant%
Without loss of generality, we can assume that \(\beta \ge 1\). By density argument on \(W^{1, p} (\Omega, \Cset)\), we can assume that \(u \in C^\infty (\Bar{\Omega})\).
We first observe that, by the classical trace theory and the diamagnetic inequality, 
\begin{equation}
  \label{eq_ieNgi9uxoo0ai}
  \int_{\partial \Omega}
  \abs{u}^p
  \le
  \C
  \Biggl(
  \int_{\Omega}
  \bigabs{\nabla_A U}^p + \abs{U}^p \Biggr)^\frac{1}{p}
  \Biggl(\int_{\Omega} \abs{U}^p\Biggr)^{1 - 1/p}
  .
\end{equation}
Since \(\partial \Omega\) is  compact and $\Omega$ is of class $C^1$, there exists maps \(\psi_i : B (0, 1) \times (-1, 1) \to \Rset^d\) that are diffeomorphism on their image such that \(\psi_i (B (0, 1) \cap (0, 1)) = \psi_i (B (0, 1) \times (0, 1)) \cap \Omega\),
\(\psi_i (B (0, 1) \times \{0\}) = \psi (B (0, 1)) \cap \partial \Omega\), \(\partial \Omega \subset \bigcup_{i = 1}^\ell \psi_i (B (0, 1/2) \times\{0\})\) and for every \(i \in \{1, \dotsc, \ell\}\), 
\(\psi_i (B (0, 1) \times \{0\})\) is geodesically convex.
Applying \cref{proposition_Bdd_dA_trace_ball} for  \(i \in \{1, \dotsc, \ell\}\), we have 
\begin{multline*}
  \iint\limits_{B (0, 1) \times B (0, 1)}
  \frac{\Bigabs{
      e^{
        i\,\intpot{\psi_i^*\! A^{\shortparallel}} ( x, y)) 
      } 
      ( u (\psi_i (y, 0)) 
      - 
       u (\psi_i (x, 0))
    }^p}{\abs{y - x}^{d + p - 1}}
  \dif x \dif y\\[-1em]
  \le
  \C
  \int\limits_{B (0, 1) \times [0, 1]}
  {\bigabs{\nabla_A ( U \compose \psi_i)}^p +\beta^\frac{p}{2} \abs{ U \compose \psi_i}^p}.
\end{multline*}
Since $\beta \ge 1$, in view of \cref{lemma_covariant_derivative_pullback} and \cref{lemma_fract_domain_boundary_local}, this implies 
that 
\begin{multline}
  \label{eq_mohs2Yee3thah}
  \iint\limits_{\psi_i (B (0, 1) \times \{0\})  \times \psi_i (B (0, 1) \times \{0\})}
  \hspace{-2em}
  \frac{\Bigabs{
      e^{
        i\,\intpot[\partial \Omega]{A^{\shortparallel}} (x, y) 
      } 
      u (y) 
      - 
      u (x)
    }^p}{d_{\partial \Omega} (y, x)^{d + p - 1}}
  \dif x \dif y\\
  \le
  \C\Biggl(\;
  \int\limits_{\psi_i (B (0, 1) \times (0, 1))}
  \hspace{-2em}
  {\bigabs{\nabla_A U}^p +\beta^\frac{p}{2} \abs{U }^p}
  + \beta^\frac{p - 1}{3} \int\limits_{\psi_i (B (0, 1) \times \{0\})}
  \abs{u}^p\Biggr).
  \end{multline}
We have then by Young's inequality and by \eqref{eq_ieNgi9uxoo0ai}
\begin{equation}
\label{eq_ReisoghieDee9loiz8x}
\begin{split}
 \beta^\frac{p - 1}{3} \int\limits_{\psi_i (B (0, 1) \times \{0\})}
  \abs{u}^p
  &\le \beta^\frac{p - 1}{2} \int\limits_{\psi_i (B (0, 1) \times \{0\})}
  \abs{u}^p\\
  &\le \C
  \int\limits_{\psi_i(B (0, 1) \times \{0\})}
  \bigabs{\nabla_A U}^p 
  + 
  \bigl(1 + \beta^\frac{p}{2}\bigr) \abs{U}^p.
\end{split}
\end{equation}
The conclusion follows by summing for \(i \in \{1, \dotsc, \ell\}\) the estimates resulting the combination of \eqref{eq_mohs2Yee3thah} and \eqref{eq_ReisoghieDee9loiz8x}.
\end{proof}

Concerning the extension, we have 

\begin{proposition}
  \label{proposition_Bdd_dA_extension_domain}
  Let \(d  \ge 1\), let \(1 \le p  <  +\infty\) and let \(\Omega \subset \Rset^d\) be a bounded domain of class $C^1$.
  There exists a constant \(C_{\Omega, p}> 0\) such that if \(A \in C^1 (\Bar{\Omega}, \bigwedge^1 \Rset^{d})\), if \(\norm{dA}_{L^\infty (\Omega)} \le \beta\) and if \(u \in W^{1 - 1/p, p} (\partial \Omega)\), then there exists \(U \in W^{1 , p} (\Omega) \cap C^\infty (\Omega, \Cset)\) such that \(\trace_{\partial \Omega} U = u\) and 
  \begin{equation*}
    \int_{\Omega}
    \bigabs{\nabla_A U }^p
    \le 
    C_{\Omega, p}
    \Biggl(
    \iint\limits_{\substack{(x, y) \in \Omega \times \Omega\\
        d_{\partial \Omega} (y, x) \le \inj_{\partial \Omega}}}
    \frac{\bigabs{
        e^{
          i\,\intpot[\partial \Omega]{A} (x, y) 
        } 
        u (y) 
        - 
        u (x)
      }^p}{\abs{y - x}^{d + p - 1}}
    \dif x \dif y
    + \bigl(1 + \beta^\frac{p - 1}{2}\bigr) \int_{\partial \Omega} \abs{u}^p
    \Biggr)
  \end{equation*}
  and 
  \begin{equation*}
    \int_{\Omega}
    \abs{U}^p
    \le
    \frac{C_{\Omega, p}}{1 +\beta^\frac{p}{2}}
    \int_{\partial \Omega} \abs{u}^p.
  \end{equation*}
\end{proposition}
\begin{proof}
Since \(\partial \Omega\) is  compact and $\Omega$ is of class $C^1$, there exists maps \(\psi_i : B (0, 1) \times (-1, 1) \to \Rset^d\) that are diffeomorphism on their image which is geodesically convex such that \(\psi_i (B (0, 1) \cap [0, 1)) = \psi_i (B (0, 1) \times (0, 1)) \cap \Omega\), \(\psi_i (B (0, 1) \times \{0\}) = \psi_i (B (0, 1)) \cap \partial \Omega\) and \(\partial \Omega \subset \bigcup_{i = 1}^\ell \psi_i (B (0, 1/2) \times \{0\}\).
Moreover, there exist smooth functions \(\eta_1, \dotsc, \eta_\ell\) in \(C^\infty (\Bar{\Omega})\) such that \(\supp \eta_i \subset \psi_i (B (0, 1/2) \times [0, 1/2])\) and \(\sum_{i = 1}^\ell \eta_i = 1\) in \(\Omega\).

By \cref{proposition_Bdd_dA_extension_local}, for every \(i \in \{1, \dotsc, \ell\}\), there exists a function \(U_i \in W^{1, p}_{\psi_i^*\! A} (B (0, 1/2) \times [0, 1/2]) \cap C^\infty (B (0, 1/2) \times [0, 1/2])\) such that \(\trace U_i = u \compose \psi_i\) on \(B (0, 1/2) \times \{0\}\) and moreover, 
\begin{multline*}
  \int_{B (0, 1/2) \times [0, 1/2]} 
  \abs{\nabla_{\psi_i^*\! A} U_i}^p
  \le 
 C  \iint\limits_{B (0, 1) \times B (0, 1)}
  \frac{\Bigabs{
      e^{
        i\,\intpot{\psi_i^*\! A^{\shortparallel}} (x, y) 
      } 
      u (\psi_i (y, 0)) 
      - 
      u (\psi_i (x, 0))
    }^p}{\abs{y - x}^{d + p - 1}}
  \dif x \dif y\\
  + 
 C  \beta^\frac{p - 1}{2} 
  \int_{B (0, 1)} \abs{u (\psi_i (x, 0))}^p\dif x
\end{multline*}
and 
\[
\int_{B (0, 1/2) \times [0, 1/2]} 
\abs{U_i}^p
\le C \int_{B (0, 1)}
\abs{u \compose \psi_i}^p.
\]
Using \cref{lemma_covariant_derivative_pullback} and \cref{lemma_fract_domain_boundary_local}, we derive that 
\begin{multline*}
\int_{\psi_i (B (0, 1/2) \times [0, 1/2])} 
\abs{\nabla_A (U_i \compose \psi_i^{-1}) }^p \\
\le 
C \int_{\psi_i(B (0, 1)) \times \psi_i(B (0, 1))}
\frac{\bigabs{
    e^{
      i\,\intpot[\partial \Omega]{A^\shortparallel} (x, y) 
    } 
    u (y) 
    - 
    u (x)
  }^p}{\abs{y - x}^{d + p - 1}}
\dif x \dif y
+
\beta^\frac{p - 1}{2}
\int_{\psi_i (B (0, 1)\times \{0\})} \abs{u}^p
\end{multline*}
and 
\[
\int_{\psi_i (B (0, 1/2) \times [0, 1/2])} 
\abs{U_i}^p
\le \int_{\psi_i(B (0, 1) \times \{0\})}
\abs{u}^p.
\]
We define now \(U \defeq \sum_{i = 1}^\ell \eta_i U_i \compose \psi_i^{-1}\).
We have by the Leibnitz rule for covariant derivatives
\[
\nabla_A U = \sum_{i = 1}^\ell \bigl((U_i \compose \psi_i^{-1}) D \eta_i + \eta_i \nabla_A (U_i \compose \psi_i^{-1})\bigr),
\]
and we conclude.
\end{proof}

\Cref{proposition_Bdd_dA_trace_domain} and \cref{proposition_Bdd_dA_extension_domain}
imply the following semi-classical estimates.

\begin{proposition}
  \label{proposition_semiclassical}
  For every \(A \in C^1 (\Bar{\Omega}, \bigwedge^1 \Rset^d)\), there exists  a positive constant $C$ such that for every 
  \(\varepsilon > 0\), we have 
  \begin{multline*}
    C^{-1} \Biggl(
    \iint\limits_{\substack{(x, y) \in \Omega \times \Omega\\
        d_{\partial \Omega} (y, x) \le \inj_{\partial \Omega}}}
    \frac{\varepsilon^p\,\bigabs{
        e^{
          i\,\intpot[\partial \Omega]{A} (x, y)/\varepsilon 
        } 
        u (y) 
        - 
        u (x)
      }^p}{\abs{y - x}^{d + p - 1}}
    \dif x \dif y + \bigl(\varepsilon^p + \varepsilon^{\frac{p + 1}{2}}\bigr) \int_{\partial \Omega} \abs{u}^p\Biggr)\\
    \le 
    \inf \biggl\{ \int_{\Omega} \abs{\varepsilon D U + i A U}^p + \bigl(\varepsilon^p  + \varepsilon^\frac{p}{2} \bigr) \abs{U}^p
    \st U \in W^{1, p}_A(\Omega) \text{ and }  \trace_{\partial \Omega} U = u\biggr\}\\
    \le 
    C \Biggl(
    \iint\limits_{\substack{(x, y) \in \Omega \times \Omega\\
        d_{\partial \Omega} (y, x) \le \inj_{\partial \Omega}}}
    \frac{\varepsilon^p\,\bigabs{
        e^{
          i\,\intpot[\partial \Omega]{A} (x, y) / \varepsilon
        } 
        u (y) 
        - 
        u (x)
      }^p}{\abs{y - x}^{d + p - 1}}
    \dif x \dif y+ \bigl(\varepsilon^p  + \varepsilon^\frac{p + 1}{2}\bigr) \int_{\partial \Omega} \abs{u}^p\Biggr).
  \end{multline*}
\end{proposition}

\section{Interpolation of magnetic spaces}

We define for every \(p  \in [1, +\infty)\) and \(\gamma \in (0, + \infty)\), the functional space \cite{Triebel_1978}*{Definition 1.8.1/1}
\begin{multline*}
    \mathfrak{W}^{1, p}_{A, \gamma} (\Rset^d)
  =
    \Bigl\{ 
    U : (0, + \infty) \to (W^{1, p}_A \bigl(\Rset^d) + L^p (\Rset^d)\bigr)      \st 
    \\ 
        U \text{ is weakly differentiable and }
        \norm{U}_{\mathfrak{W}^{1, p}_{A, \gamma} (\Rset^d)} < + \infty
        \Bigr\},        
\end{multline*}
where
\[
  \norm{U}_{\mathfrak{W}^{1, p}_{A, \gamma} (\Rset^d)}
  \defeq 
    \biggl(
      \int_0^{+\infty} 
            \bigl(\norm{U (t)}_{W^{1, p}_A ( (\Rset^d)}^p + \norm{U' (t)}_{L^p (\Rset^d)}^p\bigr)
            \,
            t^{\gamma}
        \dif t \biggr) .
\]

For every \(T \in (0, + \infty)\) and \(U \in \mathfrak{W}^{1, p}_{A, \gamma} (\Rset^d)\), one has \(U \in C ([0, T], W^{1, p}_A (\Rset^d) + L^p (\Rset^d))\) \cite{Triebel_1978}*{Lemma 1.8.1}.
In particular, the corresponding trace space can be defined by \cite{Triebel_1978}*{Definition 1.8.1/2}
\begin{equation}
\label{eq_def_Trace_V}
    \mathfrak{T}^{1, p}_{A, \gamma}
  \defeq
    \bigl\{ 
        U (0) 
      \st 
        U \in \mathfrak{W}^{1, p}_{A, \gamma} (\Rset^d) 
    \bigr\}.
\end{equation}
By a classical result in interpolation theory (see for example \cite{Triebel_1978}*{Theorem 1.8.2}), we have if \(p \in [1, +\infty)\) and \(s \in (0, 1)\),
\begin{equation}
\label{eq_interpolation_T}
    \mathfrak{T}^{1, p}_{A, (1 - s)p - 1} (\Rset^d)
  = 
    \bigl(W^{1, p}_A (\Rset^d), L^p (\Rset^d)\bigr)_{s, p},
\end{equation}
where the right-hand side denotes the real interpolation of order \(s\) and exponent \(p\) between the spaces \(W^{1, p}_A (\Rset^d)\) and \(L^p (\Rset^d)\) \cite{Triebel_1978}*{Definition 1.3.2}.

In order to characterize the trace space, we rely on the following equivalence, whose non-magnetic counterpart is classical \cite{Triebel_1978}*{Lemma 2.9.1/2}

\begin{lemma}
\label{lemma_V_W}
Let \(d \ge 1\), \(0 < s < 1\) and \(1 \le p < +\infty\).
If \(A \in C (\Rset^d, \Rset^d)\) and if \(dA \in L^\infty (\Rset^d, \bigwedge^2 \Rset^d)\), then 
\[
    \mathfrak{W}^{1, p}_{A, \gamma} (\Rset^d)
  =
    W^{1, p}_{\Bar{A}, \gamma} (\Rset^{d + 1}_+).
\]
\end{lemma}

Here \(\Bar{A} : \Rset^{d + 1}_+ \to \Rset^{d + 1}\) is the natural extension of \(A\), defined by \(\Bar{A} (x, t) =(A (x), 0)\).
The equality of \cref{lemma_V_W} is understood under the identification \(U (t) (x) = U (x, t)\).

\begin{proof}[Proof of \cref{lemma_V_W}]
\resetconstant
We assume first that \(U \in W^{1, p}_{\Bar{A}, s} (\Rset^{d + 1}_+)\). By Fubini's theorem, we have for almost every \(t\),  \(U (t) \in W^{1, p} (\Rset^d)\).
If now \(\theta \in C^\infty_c ((0, +\infty))\), we have for every \(\varphi \in C^\infty_c (\Rset^d)\),
\begin{equation}
\label{eq_Xiu9aghaej}
  \int_{\Rset^d} \int_0^{+\infty} \theta' \varphi \,U
  =
  -\int_{\Rset^d} \int_0^{+\infty} \theta \, \varphi \,U'
\end{equation}
and thus, in \(W^{1, p}_A (\Rset^d) + L^p (\Rset^d)\),
\[
   \int_0^{+\infty} \theta' \varphi \,U
  =
  -\int_0^{+\infty} \theta \,\varphi\, U' .
\]
We finally have
\begin{equation*}
      \int_0^{+\infty} 
            \bigl(\norm{U (t)}_{W^{1, p}_A (\Rset^d)}^p + \norm{U' (t)}_{L^p (\Rset^d)}^p\bigr) \, t^\gamma
        \dif t
    \le
      \C
      \iint\limits_{\Rset^{d + 1}_+}
        \bigl(\abs{\nabla_{\Bar{A}} U (x, t)}^p + \abs{U (x, t)}^p\bigr) \, t^\gamma \dif t \dif x.
\end{equation*}

Conversely, if \(U \in \mathfrak{W}^{1, p}_{A, (1 - s)p - 1}\), then \eqref{eq_Xiu9aghaej} holds and similarly,
\[
  \int_0^{+\infty} \theta \,(\operatorname{div} \varphi - iA \cdot \varphi) \, U (t)
  =
  -\int_0^{+\infty} \theta \, \varphi \, \nabla_A U (t) 
\]
Hence, by the density of tensor products, we obtain that \(U \in W^{1, 1}_\mathrm{loc} (\Rset^{d + 1}_+)\) and \(\nabla_A U (t, x) = (\nabla_A (U (t)) (x), U' (t))\).
\end{proof}

We obtain from the previous results the following characterization of the spaces by interpolation.

\begin{theorem}
  \label{theorem_interpolation}
  Let \(d \ge 1\), \(0 < s < 1\) and \(1 \le p < +\infty\).
  If \(A \in C (\Rset^d, \Rset^d)\) and if \(dA \in L^\infty (\Rset^d, \bigwedge^2 \Rset^d)\), then 
\[
    W^{s, p}_{A}
  = 
    \bigl(W^{1, p}_A (\Rset^d), L^p (\Rset^d)\bigr)_{s, p},
  \]
\end{theorem}
\begin{proof}
  By \cref{theorem_trace_halfplane_weight}, \(W^{s, p}_A (\Rset^d)\) is the trace space of \(W^{1, p}_{A, (1 - s)p - 1} (\Rset^d)\). 
  By \cref{lemma_V_W}, this latter space coincides with \(\mathfrak{W}^{1, p}_{A, (1 - s)p - 1 }(\Rset^d)\)  whose trace space \(\mathfrak{T}^{s, p}_A (\Rset^d)\) defined in \eqref{eq_def_Trace_V} is the required interpolation space by \eqref{eq_interpolation_T}.
\end{proof}

\section{Extension from a half-space}
Finally, we obtain a result about the extension from half-space to the whole space of functions in magnetic Sobolev spaces.
Set,  for \(\gamma \in \Rset\),
\[
W^{1, p}_{A, \gamma} (\Rset^{d + 1})
\defeq
\Bigl\{
U \in W^{1, 1}_{\mathrm{loc}} (\Rset^{d + 1}) 
\st 
\norm{U}_{W^{1, p}_{A, \gamma} (\Rset^{d + 1})} < + \infty\Bigr\},
\]
where
\[
\norm{U}_{W^{1, p}_{A, \gamma} (\Rset^{d + 1})}
\defeq
\biggl(\iint\limits_{\Rset^{d + 1}_+} 
\bigl(
\abs{\nabla_A U (x, t)}^p
+ \abs{U (x, t)}^p
\bigr)
\, 
\abs{t}^\gamma \dif x \dif t \biggr)^\frac{1}{p}.
\]

\begin{theorem}
  \label{theorem_extension_hetero}
  Let \(d \ge 1\), \(-1 < \gamma < p - 1\) and \(1 \le p <+\infty\).
  There exists a constant \(C > 0\) such that for every \(A \in C^1 (\Rset^{d + 1}, \bigwedge^1 \Rset^{d + 1})\) such that \(dA\) is bounded and every \(U \in W^{1, p}_{A, \gamma} (\Rset^{d + 1}_+ )\), there exists \(\Bar{U} \in W^{1, p}_{A, \gamma} (\Rset^{d + 1})\) such that 
  \(\Bar{U} = U\) on \(\Rset^{d + 1}_+\). 
  Moreover, if \(\beta \ge \norm{dA}_{L^\infty(\Rset^{d + 1})}\), 
  \begin{equation*}
    \iint\limits_{\Rset^{d + 1}} 
    \bigl(\abs{\nabla_A \Bar{U} (x, t)}^p 
    + 
    \beta^\frac{p}{2} \abs{\Bar{U} (x, t)}^p\bigr)
    t^{\gamma}
    \dif x
    \dif t
    \le 
    C
    \iint\limits_{\Rset^{d + 1}} 
   \bigl(
    \abs{\nabla_A {U} (x, t)}^p 
    + 
    \beta^\frac{p}{2} \abs{{U} (x, t)}^p\bigr)
    t^{\gamma}
    \,
    \dif x
    \dif t.
  \end{equation*}
\end{theorem}

\begin{proof}[Proof of \cref{theorem_extension_hetero}]
  This follows from \cref{theorem_trace_halfplane_weight} on \(\Rset^{d + 1}_+\) and \cref{proposition_Bdd_dA_extension} on \(\Rset^{d + 1}_-\) with \(s = 1- \frac{\gamma + 1}{p}\).
\end{proof}

\begin{remark}\label{rem-extension} \rm
A natural strategy to prove \cref{theorem_extension_hetero} would be to define the extension \(\Bar{U}\) by reflection: for every \((x, t) \in \Rset^d \times (-\infty, 0)\), we would define \(\Bar{U} (x, t) = U (x, -t)\). The computation of the covariant derivative would give
\[
\nabla_A \Bar{U} (x, t) = R \nabla_A U (x, -t)
+ i \bigl(A (x, t) - R A (x - t)\bigr) U (x, -t),
\]
where \(R\) is the orthogonal reflection with respect to 
the hyperplane \(\Rset^d \times \{0\}\).
The approach would thus only work when \(A\) is invariant under the pull-back by \(R\); this would imply the same invariance of the magnetic field \(dA\) and implies that even up to gauge transformations, it is not possible to cover a significant class of magnetic fields. 
\end{remark}

\appendix

\section{Alternative magnetic Gagliardo seminorms}

A fractional Gagliardo seminorm defined by an integral involving the mid-point has been proposed in the litterature
(see \citelist{\cite{dAvenia_Squassina_2018}*{\S 2}\cite{Nguyen_Pinamonti_Squassina_2018}\cite{Squassina_Volzone_2016}\cite{Pinamonti_Squassina_Vecchi_2017}\cite{Wang_Xiang_2016}*{\S 2}\cite{Liang_Repovs_Zhang_2018}*{\S 2}\cite{Binlin_Squassina_Xia_2018}*{\S 2}\cite{Ambrosio_2018}\cite{Ambrosio_dAvenia_2018}\cite{Fiscella_Pinamonti_Vecchi_2017}\cite{Ichinose_2013}}):
\begin{equation}
  \label{eq_def_mag_Gagliardo_mid}
  \biggl(
  \iint\limits_{\Rset^d \times \Rset^d}
  \frac
  {\bigabs{e^{i\,A (\tfrac{x + y}{2})[y - x]} u (y) - u (x)}^p}
  {\abs{y - x}^{d + sp}}
  \dif x
  \dif y
  \biggr)^\frac{1}{p};
\end{equation}
another natural candidate could be the integral involving boundary points
\begin{equation}
  \label{eq_def_mag_Gagliardo_boundary}
  \biggl(
  \iint\limits_{\Rset^d \times \Rset^d}
  \frac
  {\bigabs{e^{\frac{i}{2}\,(A(x)[y - x] + A (y)[y - x])} u (y) - u (x)}^p}
  {\abs{y - x}^{d + sp}}
  \dif x
  \dif y
  \biggr)^\frac{1}{p}.
\end{equation}
Whereas \eqref{eq_def_mag_Gagliardo} enjoyed a gauge-invariance property, this is not the case for \eqref{eq_def_mag_Gagliardo_mid} or \eqref{eq_def_mag_Gagliardo_boundary}.

The formulas \eqref{eq_def_mag_Gagliardo}, \eqref{eq_def_mag_Gagliardo_mid} and \eqref{eq_def_mag_Gagliardo_boundary} are in fact particular cases of the following general semi-norm
\begin{equation}
  \label{eq_def_mag_Gagliardo_mu}
  \norm{u}_{\dot{W}^{s, p}_{A, \mu} (\Rset^d)}
  =
  \biggl(
  \iint\limits_{\Rset^d \times \Rset^d}
  \frac
  {\bigabs{e^{i\,\mathcal{I}^\mu_{A^\shortparallel} (x, y)} u (y) - u (x)}^p}
  {\abs{y - x}^{d + sp}}
  \dif x
  \dif y
  \biggr)^\frac{1}{p},
\end{equation}
where \(\mu\) is a given finite measure on the interval \([0, 1]\) and where the potential \(\intpot[\mu]{A^\shortparallel}\) of \(A\) with respect to the measure \(\mu\) is defined by the following variant of \eqref{eq_def_intpot} 
\begin{equation*}
  \intpot[\mu]{A^\shortparallel} (x, y)
  = 
  \int_0^1 
  A \bigl((1 - t) x + t y\bigr)
  \cdot (y - x)
  \dif \mu (t).
\end{equation*}
The seminorm defined in \eqref{eq_def_mag_Gagliardo}, \eqref{eq_def_mag_Gagliardo_mid} and \eqref{eq_def_mag_Gagliardo_boundary} correspond respectively to a restriction of the Lebesgue measure \(\mu = \mathcal{L}^1 \llcorner [0, 1]\), a Dirac measure at the centre \(\mu = \delta_{1/2}\) and the average of Dirac measures at the endpoint \(\mu = \frac{\delta_0 + \delta_1}{2}\).

\begin{proposition}
  \label{proposition_Equivalent_norms}
  Let \(k \in \Nset\) and let \(\mu_1, \mu_2\) be measures on \([0, 1]\).
  If for every \(j \in \{0,\dotsc, k - 1\}\),
  \[
  \int_0^1 t^j \dif \mu_1 (t) = 
  \int_0^1 t^j \dif \mu_2 (t),
  \]
  then 
  \[
  \bigabs{
    \norm{u}_{W^{s, p}_{A, \mu_2}(\Rset^d)}
    -  \norm{u}_{W^{s, p}_{A, \mu_1}(\Rset^d)}}
  \le 
  C \norm{D^k A}_{L^\infty (\Rset^d)}^\frac{s}{k + 1}
  \norm{u}_{L^p (\Rset^d)}.
  \]
\end{proposition}

The measures corresponding to the semi-norms of \eqref{eq_def_mag_Gagliardo}, \eqref{eq_def_mag_Gagliardo_mid} and \eqref{eq_def_mag_Gagliardo_boundary} all satisfy the assumptions of \cref{proposition_Equivalent_norms} with \(k = 0\), \(k = 1\) and \(k = 2\), and the heterogeneous spaces coincide thus as soon as either \(A\) is bounded or its second or first derivative is bounded. 
This means in particular that if \(k = 1\), the estimates in the present work involving \(\norm{dA}_{L^\infty}\) and our semi-norm given by \eqref{eq_def_mag_Gagliardo}, have counterparts involving \(\norm{D A}_{L^\infty}\) and either \eqref{eq_def_mag_Gagliardo_mid} and \eqref{eq_def_mag_Gagliardo_boundary}; the latter quantities are not gauge invariant.

In the particular case where \(A\) is an affine function, then \cref{proposition_Equivalent_norms} holds with \(k = 2\) and \(D^2 A = 0\); the norms defined by \eqref{eq_def_mag_Gagliardo}, \eqref{eq_def_mag_Gagliardo_mid} and \eqref{eq_def_mag_Gagliardo_boundary} are then identical.

Higher moment identities in the assumption of \cref{proposition_Equivalent_norms} induce lower powers in the dependence on derivatives of \(A\), which can be relevant in semi-classical analyses.
The exponent in the moment condition can be increased by using a measure \(\mu\) such that more moment coincide. For instance setting \(\mu = \frac{1}{6} \delta_0 + \frac{2}{3} \delta_{1/2} + \frac{1}{6} \delta_1\), corresponding to Simpson's quadrature rule, would give estimates of \eqref{eq_def_mag_Gagliardo} up to a term of the order \(\norm{D^k A}^{\frac{sp}{k + 1}}\) for \(k \in \{0, \dotsc, 3\}\).

\begin{proof}[Proof of \cref{proposition_Equivalent_norms}]
  \resetconstant
  We have for every \(x, y \in \Rset^d\), by the triangle inequality
  \begin{multline*}
    \Bigabs{
      \bigabs{e^{i\,\intpot[\mu_2]{A^\shortparallel} (x, y)} u (y) - u (x)}
      - \bigabs{e^{i\,\intpot[\mu_1]{A^\shortparallel} (x, y)} u (y) - u (x)}
    }\\
    \le
    \bigabs{e^{i\,(\intpot[\mu_2]{A^\shortparallel} (x, y) - \intpot[\mu_2]{A^\shortparallel} (x, y)} u (y)}
    =\bigabs{e^{i\,\mathcal{I}^{\mu_2 - \mu_1}_{A^\shortparallel} (x, y)} u (y)}.
  \end{multline*}
  By integration and Minkowski's inequality, this yields
  \begin{multline*} 
    \Biggl\lvert
    \biggl(
    \iint\limits_{\Rset^d \times \Rset^d}
    \frac
    {\bigabs{e^{i\,\intpot[\mu_2]{A^\shortparallel} (x, y)} u (y) - u (x)}^p}
    {\abs{y - x}^{d + sp}}
    \dif x
    \dif y
    \biggr)^\frac{1}{p}\\
    -
    \biggl(
    \iint\limits_{\Rset^d \times \Rset^d}
    \frac
    {\bigabs{e^{i\,\intpot[\mu_1]{A^\shortparallel} (x, y)} u (y) - u (x)}^p}
    {\abs{y - x}^{d + sp}}
    \dif x
    \dif y
    \biggr)^\frac{1}{p}\Biggr\rvert \\ 
    \le \biggl(\iint\limits_{\Rset^d \times \Rset^d} 
    \frac{\bigabs{e^{i\,\mathcal{I}^{\mu_2 - \mu_1}_{A^\shortparallel} (x, y)} - 1}^p \abs{u (y)}^p}{\abs{y - x}^{d + sp}} \dif x \dif y \biggr)^\frac{1}{p}.
  \end{multline*}
  
  By our assumption on the moments of the measures \(\mu_1\) and \(\mu_2\), there exists a constant \(\Cl{cst_yadee7TeeV}\) depending only on the measure \(\mu_2 - \mu_1\) such that 
  \[
  \bigabs{\intpot[\mu_2 - \mu_1]{A^\shortparallel} (x, y)}
  \le \Cr{cst_yadee7TeeV}
  \norm{D^k A} \abs{y - x}^{k + 1},
  \]
  and thus for every \(y \in \Rset^d\)
  \[
  \int_{\Rset^d}
  \frac{\bigabs{e^{i\,\intpot[\mu_2 - \mu_1]{A^\shortparallel} (x, y)} - 1}^p }{\abs{y - x}^{d + sp}} \dif x
  \le 
  \C \norm{D^k A}_{L^\infty}^\frac{sp}{k + 1};
  \]
  the conclusion then follows.
\end{proof}

\end{document}